\documentclass[leqno,titlepage,11pt]{article}

\usepackage[a4paper, hdivide={4.3cm,,1.7cm},vdivide={3.48cm,,3.48cm}]{geometry}

\usepackage{amssymb}
\usepackage{amsmath}
\usepackage{amsthm}
\usepackage{enumerate}
\usepackage{amstext}
\usepackage{graphicx}			
\usepackage{hhline}
\usepackage{psfrag}
\usepackage{float}
\usepackage[dvips]{epsfig}
\usepackage{rotating}
\usepackage[utf8]{inputenc}
\usepackage[T1]{fontenc}

\usepackage{comment}

\usepackage{pstricks}

\includecomment{comment}

\includecomment{explainDef}

\includecomment{explainText}

\includecomment{explainVor}

\includecomment{grothaus}

\includecomment{profireader}

\excludecomment{scherz}

\usepackage{diagrams}

\makeatletter\@addtoreset{equation}{section}\makeatother

\allowdisplaybreaks

\DeclareSymbolFont{msbm}{U}{msb}{m}{n}
\DeclareMathSymbol{\N}{\mathalpha}{msbm}{'116}
\DeclareMathSymbol{\R}{\mathalpha}{msbm}{'122}

\author{Benedict Baur}
\title{Core property of smooth contractive embeddable functions for an elliptic operator}

\begin{document}

\newtheoremstyle{theorem}%
{3pt}
{3pt}
{\itshape}
{}%
{\bfseries}
{}
{.5em}
{}%

\newtheoremstyle{lemma}%
{3pt}
{3pt}
{}
{}%
{\bfseries}
{}
{.5em}
{}%

\newtheoremstyle{definition}%
{3pt}
{3pt}
{}
{}%
{\bfseries}
{}
{.5em}
{}%

\newtheoremstyle{remark}%
{3pt}
{3pt}
{}
{}%
{}
{}
{.5em}
{}%

\theoremstyle{theorem} 

\newtheorem{theorem}{Theorem}[section] 

\theoremstyle{lemma} 
\newtheorem{lemma}[theorem]{Lemma}
\newtheorem{corollary}[theorem]{Corollary}

\theoremstyle{definition} 
\newtheorem{definition}[theorem]{Definition}
\newtheorem{assumption}[theorem]{Assumption}

\theoremstyle{remark} 
\newtheorem{remark}[theorem]{Remark}

\newcommand{\Laplace}{\Delta}
\newcommand{\Om}{\Omega}
\newcommand{\Omclo}{\overline{\Omega}}

\newcommand{\Czero}[1]{C_0^{0}(#1)}
\newcommand{\Czeroalpha}[1]{C^{0,\alpha}(#1)}
\newcommand{\Czerozero}[1]{C^{0}(#1)}
\newcommand{\Czeroalphazero}[1]{C^{0,\alpha}(#1)}

\newcommand{\Ctwo}[1]{C^{2}(#1)}
\newcommand{\Ctwoalpha}[1]{C^{2,\alpha}(#1)}
\newcommand{\Ctwoalphac}[1]{C_c^{2,\alpha}(#1)}

\newcommand{\Ctwoalphacompacttwo}{C^{2,\alpha}_c(\R^n)}

\newcommand{\Bplus}[1]{B^{+}_{#1} (0)}
\newcommand{\Bminus}[1]{B^{-}_{#1} (0)}
\newcommand{\Bzero}[1]{B^{0}_{#1} (0)}
\newcommand{\BRplus}{\Bplus{R}}
\newcommand{\BRplusclo}{\overline{\Bplus{R}}}
\newcommand{\BRminus}{\Bminus{R}}
\newcommand{\BRzero}{\Bzero{R}}
\newcommand{\BR}{B_R(0)}
\newcommand{\Rhalf}{\frac{R}{2}}
\newcommand{\BRhalfclo}{\Bplus{\Rhalf} \cup \Bzero{\Rhalf}}
\newcommand{\emb}{\hookrightarrow}
\newcommand{\firstn}{x^{(n-1)}}
\newcommand{\firstnx}{x^{(n-1)}}
\newcommand{\firstny}{y^{(n-1)}}
\renewcommand{\mod}[1]{|#1|}

\def\Hess{\rm H}

\newcommand{\Ctwoalphaf}{C^{2,\alpha}}
\newcommand{\Czeroalphaf}{C^{0,\alpha}}

\newcommand{\Ckalpha}{C^{k,\alpha}}
\newcommand{\Ctwof}{C^{2}}

\newcommand{\floone}{F^{1}}

\newcommand{\pos}{(\firstn+g(\firstn)x_n)}
\newcommand{\dx}[1]{\frac{d}{dx_{#1}}}
\newcommand{\pdx}[1]{\partial_{x_{#1}}}
\newcommand{\pdy}[1]{\partial_{y_{#1}}}

\newcommand{\pdxo}{\pdx{}}

\newcommand{\supr}[1]{\underset{#1}{\sup}}
\newcommand{\infi}[1]{\underset{#1}{\inf}}

\newcommand{\sumfromto}[2]{\overset{#2}{\underset{i=#1}{\sum}}}
\newcommand{\sumfromtoj}[2]{\overset{#2}{\underset{j=#1}{\sum}}}
\newcommand{\mysum}[2]{\overset{#2}{\underset{#1=1}{\sum }}}
\newcommand{\sumij}{\underset{i,j}{\sum}}
\newcommand{\sumi}{\underset{i}{\sum}}
\newcommand{\sumj}{\underset{j}{\sum}}

\newcommand{\DL}{\{u \in \Ctwoalpha{\Omclo} \ | \ u=0,\ Lu=0 \ \text{on} \ \partial \Om \}}
\newcommand{\defL}{\mysum{i,j}{n} a_{ij} \pdx{i}\pdx{j} u + \sumfromto{1}{n} b_{i} \pdx{i} u + c u}
\newcommand{\weightedNorm}[3]{\| #1 \|_{#2}^{#3}}  
\newcommand{\weightedNormDef}[4]{\supr{x \in #2 } (d(x, #3))^{#4} | #1(x) |}
\newcommand{\weightedTotalNorm}{\| u \|'_{\Ctwoalpha{\Om}}}
\newcommand{\weightedTotalNormBound}{\| u \|'_{\Ctwoalpha{\Om \cup T}}}

\newcommand{\norm}[2]{\| #1 \|_{#2}}

\newcommand{\interior}[1]{\overset{\circ}{#1}}
\newcommand{\clo}[1]{\overline{#1}}
\newcommand{\diracdelta}[2]{\delta_{#1}({#2})}
\newcommand{\dualitymap}[2]{\langle #1 , #2 \rangle}
\newcommand{\eucp}[2]{#1 \bullet #2}

\newcommand{\traceDiffOp}[2]{tr(#1 \Hess \ #2)}

\newcommand{\limit}[1]{\overset{#1}{\longrightarrow}}
\newcommand{\limk}{\limit{k \rightarrow \infty}}
\newcommand{\limn}{\limit{n \rightarrow \infty}}
\newcommand{\limi}{\limit{i \rightarrow \infty}}
\newcommand{\limx}{\limit{x \rightarrow x_0}}

\newcommand{\half}[1]{\frac{#1}{2}}

\newcommand{\epsi}{\varepsilon}
\newcommand{\whatthisis}{\text{thesis}} 
\newcommand{\darstellung}{\text{presentation}} 
\newcommand{\ellip}{\lambda} 
\newcommand{\coeffbound}{\Lambda} 
\newcommand{\dualitysetnull}[1]{\textit{F}(#1)}
\newcommand{\dualityset}{\dualitysetnull{u}}
\newcommand{\dualitysetx}{\dualitysetnull{x}}
\newcommand{\bezInvFunc}{\text{Inverse function theorem }}
\newcommand{\isomorph}{\simeq}

\newcommand{\hoelder}{\text{hölder }}


\newcommand{\harm}[1]{\overline{#1}}
\newcommand{\boundf}{g} 
\newcommand{\mtilde}{\widetilde}
\newcommand{\Id}{\textbf{I}}
\newcommand{\supp}{\text{supp}}



\thispagestyle{empty}
\begin{verbatim}

\end{verbatim}

\begin{center}
\Large{\textsc {University of Kaiserslautern}}\\
\Large{\textsc {Department of Mathematics}} \\
\Large{\textsc {Functional Analysis and Stochastic Analysis Group}}
\end{center}
\begin{verbatim}









\end{verbatim}

\begin{center}
\textbf{\huge{\rule[-.1cm]{0cm}{0.1cm} Core property of smooth contractive embeddable functions for an elliptic differential operator}}
\end{center}
\begin{verbatim}







\end{verbatim}
\begin{center}
\textit{{Diploma Thesis}}\\[0.5ex]
\textit{{by}}\\[0.5ex]
\textit{{Benedict Baur}}
\end{center}
\begin{verbatim}



\end{verbatim}
\begin{center}
\textit{{Supervisor: Prof. Dr. M. Grothaus}}
\end{center}
\begin{verbatim}



\end{verbatim}


\begin{verbatim}
 
\end{verbatim}

\begin{center}
\textit{{March, 2010}}
\end{center}

\newpage
\tableofcontents
\newpage
\section{Introduction}
\begin{explainText}
This $\whatthisis$ is concerned with two assumptions (see \ref{assSemi} and \ref{assCont} below) made in \cite{Grothaus}, these lead to a representation of a solution to the Cauchy-Dirichlet problem for elliptic differential operators. The representation consists of an iterated sequence of integrals over elementary functions in terms of the coefficients of the differential operator.
\end{explainText}
Throughout this $\whatthisis$ let $n$ be the dimension of the space and $\alpha$ some positive number less or equal to $1$.
Let $\Om$ be a bounded region, $a_{ij}$, $b_{i}$, $c \in \Czeroalpha{\Omclo}$.
Assume that:
\begin{enumerate}
\item The matrices $(a_{ij}(x))$ are uniformly elliptic, i.e. there exists an $\ellip>0$ such that $\sumij a_{ij}(x) \xi_i \xi_j \geq \ellip \norm{\xi}{2}^2$ for all $x \in \Om$ and $\xi \in \R^n$,$\ellip > 0 $
\item $ \mod{a_{ij}(x)} \le \coeffbound$, $\mod{b_i(x)} \le \coeffbound$,$\mod{c(x)}\le \coeffbound$ for some $\coeffbound < \infty$
\end{enumerate}

For $u \in \Ctwo{\Omclo}$ define
\begin{eqnarray} L u := \defL \label{defG} 
\end{eqnarray}

We consider the Cauchy-Dirichlet problem in $\Om$. For $u_0 \in \Ctwo{\Omclo}$ find 
$$ u:[0,\infty) \times \Omclo \rightarrow \R  $$ with $u(t,\cdot) \in \Ctwo{\Omclo}$ and
$u(\cdot) \in C^1([0,\infty),\Czerozero{\Omclo})$ fulfilling:
 
\begin{eqnarray} 
 u'(t,x) = L u \ \forall \ t \ge 0 \nonumber \\
 u(0,\cdot) = u_0 \label{CauchyProblem} \\
 u(t,x) = 0 \ , x \ \in \ \partial \Om  \nonumber
\end{eqnarray}
In order to show the existence of a solution and to verify the representation of this solution two assumptions were made in \cite{Grothaus}:
Set
\begin{eqnarray}
D(L) = \DL \label{defDL}
\end{eqnarray}

\begin{assumption} \label{assSemi}
The closure of $(L,D(L))$ generates a strongly continuous operator semigroup on $\Czero{\Omclo}$, which solves \eqref{CauchyProblem} \ for $u_0 \in D(L)$. 
\end{assumption}
\begin{assumption} \label{assCont}
The set $\DL$  can be continuously and linearly embedded into $\Ctwoalphac{\R^n}$ such that $\| u \|_{C^0_c(\R^n)} = \| u \|_{C^0_0(\Om)}$
\end{assumption}
In other words the function space $\DL$ is a core for the generator of the resolving semigroup $(T_t)$ of \eqref{CauchyProblem}. These functions can be embedded linearly into $\Ctwoalphac{\R^n}$ such that the supremum is not increased.
For more details we refer the reader to \cite{Grothaus}.
It is important to note, that it is not sufficient for the extension operator to be bounded, it has to be a contraction.
This \whatthisis \ is organised as follows:  In the first section we investigate the regularity theory for elliptic partial differential equations, this theory is the basis for solving the resolvent equation associated to \eqref{CauchyProblem}.
In the second section we proof the assumptions for domains with $C^{4,\alpha}$ smooth boundaries and coefficients in $\Ctwoalpha{\Omclo}$.
These additional assumptions are necessary because the differential operator is transformed using the diffeomorphism which flattens out the boundary. The transformed differential operator posesses drift terms containing the second derivatives of this diffeomorphism. However to construct the embedding operator these coefficients must be in $C^{2,\alpha}(\partial \Om)$.
\section{Dirichlet problem for elliptic differential operators}

\def\LufOm{Lu = f \ \text{in} \  \Om}
In this section we consider the (stationary) Dirichlet problem for the elliptic differential operator as defined in \eqref{defG} or the associated resolvent equation respectively. 
\begin{explainVor}
The solution theory is based on maximum principles, a priori estimates on the $\Ctwoalphaf$-norm and the solvability on balls. The results stated here are based on the corresponding results for the Laplace-operator, which can be found in \cite{GilTru} in Chapter 1 and 2. In the following we present the necessary lemmas and theorems, but give not all proofs, they can be found in \cite{GilTru}. Another presentation of the theory can be found in \cite{Jost}.
\end{explainVor}

\begin{definition}[Dirichlet problem for elliptic equations] \label{DirProblemEllip} 

Let $L$ be the differential operator \eqref{defG}.
For a given $f \in C^0(\Omclo)$, $g \in C^0(\Omclo)$ find $u \in C^0(\Omclo) \cap \Ctwo{\Om} $ such that:
\begin{eqnarray} 
 \LufOm \\ \label{DirProblem}
 u = g \ \text{on} \ \partial \Om \nonumber
\end{eqnarray}

\end{definition}

\subsection{Maximum principles}
\begin{theorem}[Weak Maximum principle for c=0]{ \ } \label{WeakMax0}
\\ Let $L$ be the differential operator \eqref{defG} with $c = 0$, $u \in \Ctwo{\Om} \cap C(\clo{\Om})$, then the following statements hold:
\begin{enumerate}
\item If $Lu \geq 0 \ \ \text{in}  \ \Om$, then $ \supr{\Om} \ u = \supr{\partial \Om} \ u$
\item If $Lu \leq 0 \ \ \text{in}  \ \Om$, then $\infi{\Om} \ u = \infi{\partial \Om} \ u$
\item If $Lu = 0 \ \ \text{in}  \ \Om$, then $\supr{\Om} \ | u | = \supr{\partial \Om} \ | u |$
\end{enumerate}

\end{theorem}
The proof of this theorem can be found in \cite{GilTru} on page 31, Theorem 3.1.

\begin{corollary}{Weak maximum principle for $c \leq 0$} \label{WeakMax}
\\ Let $L$ be the differential operator \eqref{defG} with $c \le 0$, $u \in \Ctwo{\Om} \cap C(\clo{\Om})$, then the following statements hold:
\begin{enumerate}
\item If $Lu \geq 0 \ \ \text{in}  \ \Om$ then $\supr{\Om} \ u \leq \supr{\partial \Om} \ u^{+}$
\item If $Lu \leq 0 \ \ \text{in}  \ \Om$ then $\infi{\Om} \ u \geq - \supr{\partial \Om} \ u^{-}$
\item If $Lu = 0 \ \ \text{in}  \ \Om$ then $\supr{\Om} \ | u | = \supr{\partial \Om} \ |u| $
\end{enumerate}

\end{corollary}
The proof of the weak maximum principle can be found in \cite{GilTru} on page 32, Corollary 3.2.

\begin{corollary}{Uniqueness of a solution to the Dirichlet problem}
\\ There exists at most one solution in $C^0(\Omclo) \cap C^2(\Om)$ to \eqref{DirProblem}
\end{corollary}

\begin{theorem}[Strong maximum principle]{\ } \label{StrongMax} 
\\ Let $L$ be the differential operator \eqref{defG} with $c \leq 0$, $u \in C^0(\Omclo) \cap C^2(\Om)$ with $Lu = 0$. Then $u$ cannot attain a non-negative maximum, unless it is constant.
\end{theorem}
Compare Theorem 3.5, page 34 in \cite{GilTru}. The condition $\frac{c}{\lambda} < \infty$ is provided by our assumption that $c$ is bounded and the uniform ellipticity.
The maximum principles and their corollaries come from the fact, that a $C^2(\Om)$-smooth function has a negative definite Hesse-Matrix at a interior maximum, while the coefficient matrix of the second order part of the differential operator is positive definite. Therefore it is not suprising that the differential operator is dissipative, provided $c\le0$.

\begin{lemma}{Operator dissipative} \label{Dissipative}
\\ If $c \leq 0$ then the operator $(L,D(L))$ of \eqref{defG} with $D(L)$ as in \eqref{defDL} is dissipative. 
 Otherwise there exists an $\omega > 0$ such that the operator $L-\omega$ is dissipative.
\end{lemma}
\begin{proof}
{ Case 1: $c \leq 0$ }
Recall the definition of dissipative: For each $u \in D(L)$ there exists an $u' \in \dualityset$ such that $\langle u', L u \rangle \leq 0$. For $u \equiv 0$ the condition is trivial. If $u$ is not constant zero, then by the boundary conditions $|u|$ attains its maximum in the interior. Assume first that the maximum is attained for a positive value at a point $x_0 \in \Om$. Choose $u':=\diracdelta{x_0}{\cdot}$. 
Then $ \dualitymap{u'}{u} = \dualitymap{\diracdelta{x_0}{\cdot}}{u} = \diracdelta{x_0}{u} = u(x_0) = \| u \| $. Note that at a maximum $\nabla u(x_0) = 0$, $tr(A (Hu))(x_0) \leq 0 $ since $A$ is positive definit and $H$ negative definit.  
$$\dualitymap{u'}{Lu} = ( tr(A (Hu)) + \sumfromto{1}{n} b_{i} \pdx{i} u + c u)(x_0) \leq c u \leq 0$$
If the maximum $\| u \|$ is attained for $u < 0$ consider $-u$. Then $\dualitymap{-\delta_{x_0}}{Lu} = 
\dualitymap{\delta_{x_0}}{L(-u)} \le 0$ \newline 
{Case 2: $c$ arbitrary:} Since $c$ is bounded we have $ \omega := \supr{x \in \Om}{\ c} < \infty$. Apply the proof of case 1 to $L' = L - \omega$.
\end{proof}
The property dissipative can be viewn as the local analogon of the maximum principle.

\begin{theorem}[Interior Maximum estimate]{ \ } \label{IntMaxEstimate} \\
Let $L$ be the differential operator \eqref{defG} with $c \leq 0$, $u \in C^0(\Omclo)\cap C^2(\Om)$, $f \in C^0(\Omclo)$. If $Lu=f$ in a bounded domain $\Om$ then:
$$ \supr{\Om}{|u|} \leq \supr{\partial \Om}{|u|}+\frac{C}{\ellip} \norm{f}{C^0(\Omclo)} $$
Moreover the constant $C$ depends monotone increasing on the bound of the coefficients $\coeffbound$ and decreasing on the ellipticity constant $\ellip$.
\end{theorem}
The proof is based on the maximum principle and the construction of a certain dominating function. This will be done in \ref{SubSuperfunction} below, for the monotone dependence see the remark after the proof of \ref{SubSuperfunction}.

\subsection{Estimate in the interior}
In this subsection we present certain a-priori estimates to a solution. We require $f \in \Czeroalpha{\Omclo}$ and discuss the existence of solutions in $\Ctwoalpha{\Om}$ or $\Ctwoalpha{\Omclo}$. Therefore we first cite a priori estimates on the solution in those spaces. Based on these estimates and the maximum principle we derive the existence and uniqueness of solutions in the next section.
The interior estimates depend on the distance to the boundary. Therefore we introduce the following distance-weighted norms:
Let $0 < \beta$, $\Om_0 \subset \Om$, $T \subset \partial \Om$.
\begin{explainText}
\begin{enumerate}
\item $\norm{u}{C^0(\Om_0)}^{(\beta)} := \supr{x \in \Om_0} (d(x,\partial \Om_0))^\beta \mod{u(x)} $
\item $\norm{u}{C^0(\Om_0 \cup T)}^{(\beta)}:=\supr{x \in \Om_0 \cup T} (d(x,\partial \Om_0 \setminus T))^\beta \mod{u(x)}$
\end{enumerate}
The weighted norms $\weightedNorm{u}{\Czeroalpha{\Om_0}}{(\beta)}$ are defined analougsly with $|u(x)|$ replaced by the lim sup of the $\alpha$-differential quotient.
Define $$\| u \|'_{\Ctwoalpha{\Om_0}} = \weightedNorm{u}{C^0(\Om_0)}{(0)}+\weightedNorm{Du}{C^0(\Om_0)}{(1)}+\weightedNorm{D^2u}{C^0(\Om_0)}{(2)}+\weightedNorm{D^2u}{\Czeroalpha{\Om_0}}{(2+\alpha)}$$ \newline 
For a boundary portion $T$ define: \newline
$$\| u \|'_{\Ctwoalpha{\Om_0 \cup T}} = \weightedNorm{u}{C^0(\Om_0 \cup T)}{(0)}+\weightedNorm{Du}{C^0(\Om_0 \cup T)}{(1)}+\weightedNorm{D^2u}{C^0(\Om_0 \cup T)}{(2)}+\weightedNorm{D^2u}{\Czeroalpha{\Om_0 \cup T}}{(2+\alpha)}$$
I.e. as nearer a point $x$ is to the boundary the weaker the value at this point counts.
Further remarks to this norm can be found in \cite{GilTru} on page 60.
\end{explainText} 

%

\begin{lemma}{Estimate in the interior - Schauder Estimate} \label{EllipEstimateInt} \newline
If $u \in C^0(\Omclo) \cap \Ctwoalpha{\Om}$ with $Lu=f$. Then 
\begin{eqnarray}
\weightedTotalNorm{} \leq C(n,\ellip,\coeffbound,\alpha)(\| u \|_{C^0(\Omclo)} + \weightedNorm{f}{C^0(\Om)}{(2)}+\weightedNorm{f}{C^{0,\alpha}(\Om)}{(2+\alpha)})
\label{EstimateInt}
\end{eqnarray} 
If $\Om$ has a boundary portion $T$ on $\{ x_n = 0\}$ and $u=0$ on $T$ then:
\begin{eqnarray}\weightedTotalNormBound{} \leq C (\norm{u}{C^0(\Omclo)}
+\weightedNorm{f}{C^0(\Om)}{(2)}+\weightedNorm{f}{C^{0,\alpha}(\Om)}{(2+\alpha)}) 
\label{EstimateBond}
\end{eqnarray} 
Moreover the constant depends monotone decreasing on $\ellip$ and monotone increasing on $\coeffbound$.
\end{lemma}
\begin{proof}
The proof of the first estimate can be found in \cite{GilTru}, Theorem 6.2 on page 85. It is based on an estimate for operators with constant coefficient matrix and no lower-order terms, which can be found in \cite{GilTru}, Lemma 6.1, page 83.
The monotone depending of the constant $C$ is contained in the proof of the estimate for a constant coefficient matrix. It depends on the norm of the diffeomorphism which transforms the constant coefficient matrix to the idendity matrix.
The analouge estimate with boundary portion can be found in \cite{GilTru}, Lemma 6.4 , page 90.
\end{proof}

\begin{remark}
Since $\weightedNorm{f}{C^0(\Om)}{(2)}+\weightedNorm{f}{C^{0,\alpha}(\Om)}{(2+\alpha)} \le C \norm{f}{\Czeroalpha{\Omclo}}$ the right hand side of \eqref{EstimateInt} can be replaced by 
$\norm{u}{C^0(\Om)} + \norm{f}{\Czeroalpha{\Omclo}}$. The same for \eqref{EstimateBond}.
\end{remark}

\begin{corollary}{Estimate on precompact set} \label{EllipEstimateComp} \newline
Let $u$ as in \ref{EllipEstimateInt}, $\Om_0 \subset \subset \Om  $, then
$$\| u \|_{\Ctwoalpha{\Om_0}} \leq C (\norm{u}{C^0(\Omclo)} + \weightedNorm{f}{\Czeroalpha{\Om}}{(2+\alpha)}) $$
If $\Om$ has a boundary part on $\{ x_n = 0\}$. Set $T_0 = \partial \Om \cap \{ x_n = 0\}$. Let $\Om_0 \subset \Om$ with $T'=\partial \Om_0 \cap T_0$ and $T' \subset T  \subset T_0$ with dist$(\Om_0,\partial \Om \setminus T) >0$ then:
$$ \norm{u}{\Ctwoalpha{\Om_0 \cup T'}} \leq C (\norm{u}{C^0(\Omclo)} + \weightedNorm{f}{\Czeroalpha{\Om}}{(2+\alpha)})  $$
\end{corollary}
This follows by dist$(x, \partial \Om) \ge d_0 > 0$ or dist$(x, \partial \Om \setminus T)\ge d_0 > 0$ for $x \in \Om_0$ or $x \in \Om_0 \cup T'$ respectively. 

\begin{theorem}[Estimate on smooth domains]{ \ } \label{EstimateSmoothBoundary} \\
If $\partial \Om$ is $\Ctwoalphaf$-smooth and $u = 0 \ \text{on} \ \partial \Om$, $\LufOm$ then:
$$\| u \|_{\Ctwoalpha{\Omclo}} \leq C(\Om,\ellip,\coeffbound,\alpha) (\| u \|_{C^0(\Omclo)} + \|f\|_{\Czeroalpha{\Om}}) $$
\end{theorem}
The proof can be found in \cite{GilTru} page 93, Theorem 6.6.

For locally smooth boundaries we have the following local estimate:
\begin{lemma}{Local estimate on smooth boundary portion} \label{EstimateLocalSmoothBoundary}
\\ Let $\Om$ be a domain with $\Ctwoalphaf$-smooth open (w.r.t the trace topology) boundary portion $T$, $\boundf \in \Ctwoalpha{\Omclo}$, $u \in \Ctwoalpha{\Om \cup T}$ with $\LufOm$, $u = \boundf \ \text{on} \ T$. Then for each $x_0$ and each ball $B:=B_{\delta}(x_0)$ with $\delta < dist(x_0,\partial \Om \setminus T)$:
$$ \norm{u}{\Ctwoalpha{B \cap \clo{\Om} }} \leq C (\norm{u}{C^0(\Omclo)} + \norm{\boundf}{\Ctwoalpha{\Omclo}} + \norm{f}{\Czeroalpha{\Om}}) $$
\end{lemma}

\begin{remark}
It is important to note, that the norm estimate depends on the $\Ctwoalpha{\Omclo}$-norm of the boundary function at the boundary portion $T$ only and not on the boundary values of the rest of the boundary. 
Note that in \ref{EstimateLocalSmoothBoundary} the boundary function is assumed to be $\Ctwoalphaf$-smooth in the whole domain. For $\Ctwoalphaf$-smooth boundaries there exists always a $\Ctwoalphaf$-smooth continuation of a function which is $\Ctwoalphaf$ on the boundary to a $\Ctwoalphaf$-smooth function in the interior (see \ref{TraceCk}).
\end{remark}

%
%
%
%

\subsection{Perron method, Existence of Solutions}

\begin{theorem}[Solution on ball]{ \ } \label{SolBall}
\\ Let $\Om = \BR$ for some $R>0$, $\boundf \in \Ctwoalpha{\partial \Om}$, $f \in \Czeroalpha{\Omclo}$ then the problem:
\begin{eqnarray} 
 Lu = f \ \text{in}  \ \Om \label{DirBallNonhom} \\
 u = \boundf \ \text{on} \ \partial \Om \nonumber
\end{eqnarray} 
has a unique solution in $\Ctwoalpha{\Omclo}$.
\end{theorem}
\begin{proof}
By \ref{TraceCk} the function $g$ can be continued to a $\Ctwoalphaf$-smooth function (also denoted by) $g$ in the interior.

Now substitue $u$ by $u - g$. Then the problem \eqref{DirBallNonhom} is equivalent to:
\begin{eqnarray} 
 Lu = f - L g \  \text{in}  \ \Om \label{DirBallHom} \\
 u = 0 \ \text{on} \ \partial \Om \nonumber
\end{eqnarray} 
i.e we consider the problem for zero boundary values. 
The solution to \ref{DirBallNonhom} is then given by $u_0 = u + \boundf$, indeed $L u_0 = Lu+L g = (f - L g) + L g =f $. By the uniqueness of the solution of \eqref{DirBallNonhom} $u_0$ is the only solution and therefore independent of the concrete choice of the continuation of $g$. 

Consider the two Banach spaces $$X_1 = \{u \in \Ctwoalpha{\Omclo} \ | \ u = 0 \ \text{on} \ \partial \Om \}$$ $$X_2 = \{u \in \Czeroalpha{\Omclo} \ | \ u = 0 \ \text{on} \ \partial \Om \} $$
Then the operator $\Delta$ is bijective from $X_1 \rightarrow X_2$ by Corollary 4.14, page 66 in \cite{GilTru}.
Consider the family of operators $L_t = t L + (1-t) \Delta$. Then the coefficient matrix is given by
$a^t_{ij} (x) = (1-t)\delta_{ij}+t a_{ij}(x)$.
Thus: $$\sumi \sumj a^t_{ij}(x) \xi_i \xi_j = (1-t) \norm{\xi}{}^2 + t \ellip \norm{\xi}{}^2 > min \{1,\ellip \} \norm{\xi}{}^2$$
Thus the bound of the coefficients of $L_t$ is given by $max \{ 1, \coeffbound \}$.
Choosing $\tilde{\ellip} = min \{1,\ellip \}$ and $\tilde{\coeffbound} = max \{1 , \coeffbound \} $
there exists by the estimate \ref{EstimateSmoothBoundary} and the monotone dependence of the constant in (\ref{EllipEstimateInt}) a constant $C_1$, such that for every $t \in [0,1]$ and $u_t \in X_1$,$f \in X_2$ with $L_t u_t = f \ \text{in} \ \Om$:

$$ \norm{u_t}{\Ctwoalpha{\Omclo}} \le C_1(\tilde{\ellip},\tilde{\coeffbound})(\norm{u_t}{C^0(\Omclo)} + \norm{f}{\Czeroalpha{\Omclo}}) $$

Furthermore by \ref{IntMaxEstimate} there exists a constant $C_2$ with 
$$\norm{u_t}{C^0(\Omclo)} \le \frac{C_2}{\tilde{\ellip}} \norm{f}{C^0(\Omclo)} $$
thus there exists one constant $C_3 < \infty$ such that for all $t \in [0,1]$:
$$\norm{u_t}{\Ctwoalpha{\Omclo}} \le C_3 \norm{f}{\Czeroalpha{\Omclo}} $$
i.e. $\norm{L_t u}{X_2} = \norm{f}{X_2} \ge \frac{1}{C3} \norm{u_t}{X_1}$.
Thus the family of operators $L_t$ is uniformly invertible. 
Since $L_0 = \Delta$ is surjective it follows by the Method of continuity (\ref{TheoMethodCont}), that $L_1 = L$ is surjective.
\end{proof}

\begin{theorem}[Solution on a smooth boundary portion for Balls]{ \ } \label{BallPart}
\\
Let $B=\BR$. Let $\boundf \in C^0(\partial B) \cap \Ctwoalpha{T}$ for some open (w.r.t to the trace topology) possibly empty  $T \subset \partial B$.  Then there exists a unique solution of \eqref{DirProblem} in $C^0(\clo{B}) \cap \Ctwoalpha{B \cup T}$.
\end{theorem}
\begin{proof}
\def\Tray{\tilde{T}}
Define $B_2 = B_{2R}(0)$. We continue the function by a radial extension. Choose a cutoff $\eta$ for $[R-\frac{R}{4},R+\frac{R}{4}]$ in $[R-\frac{R}{2},R+\frac{R}{2}]$.
Define $g_1(x)$ := $ \eta(\mod{x}) g(R\frac{x}{\mod{x}}) $. \\
Denote by $\Tray = [0,2R] T = \{ x \ | \ R \frac{x}{\mod{x}} \in T, \ \mod{x}\le 2R \}$. Then $g_1 \in C^0(B_2) \cup \Ctwoalpha{\Tray}$.
Choose a standard dirac sequence $\varphi_k : \R^n \rightarrow \R$ with compact support in $B_{\frac{1}{k}}$ and define $g_k := \varphi_k \ast g_1$. 
Therefore:
$$ g_k \in \Ctwoalpha{\clo{B}} $$
$$ g_k \limk g_1 \text{w.r.t } \ \norm{\cdot}{C(B)}$$

By the previous theorem \ref{SolBall} there exist $u_k \in \Ctwoalpha{\clo{B}}$ with:
\begin{eqnarray}
 L u_k = f \  \text{in}  \ B \\
 u_k = \boundf_k \ \text{on} \ \partial B
\end{eqnarray} \label{SolBound}

By the 3rd estimate in \ref{WeakMax} the convergence of the boundary values implies that the sequence $(u_k)_k$ is a Cauchy sequence w.r.t to the sup-Norm in $\clo{B}$. Hence there exists $u \in C^0(\clo{B})$ with $u_k \limk u \ \text{w.r.t} \ \norm{\cdot}{C^0(\clo{B})}$. 
Thus $ u = g \ \text{on} \ \partial B$.
Since for $m,n \in \N$ we have $L u_m = L u_n = f$, we have for an arbitrary $x_0 \in B$ and some ball $B'(x_0) \subset \subset B$ by the estimate for compact subsets \eqref{EllipEstimateComp} 

$$\norm{u_m-u_n}{\Ctwoalpha{B'}} \leq C (\norm{ u_m-u_n}{C^0(\clo{B})} + \norm{L u_m - L u_n}{\Czeroalpha{\clo{B}}})$$
Hence $(u_k)_k$ is a Cauchy sequence w.r.t $\norm{\cdot}{\Ctwoalpha{B'}}$, which converges to $u$.
Thus $ L u = f \ \text{in} \ B'$ and $u \in \Ctwoalpha{B'}$.
Since $x_0$ was arbitrary we have $u \in C^0(\Omclo) \cap \Ctwoalpha{\Om}$ and $u$ solves \eqref{DirBallNonhom}.
%

%
\newcommand{\wz}{\omega_0}


\def\bPart{\Gamma}
Now let $x_0$ be any point in $T$. Since $T$ is open there exists $\delta > 0$ with $B_\delta(x_0) 
\cap T \subset \subset T$ and $B_\delta(x_0) \subset \subset \Tray$.
Denote by $B_2 = B_{\half{\delta}}(x_0) \cap B$, $B_1 = B_\delta(x_0) \cap B$, $T_2 := \partial B_2 \cap \partial B$.
By \ref{ConvCkEstimate} we have:
$$ \norm{g_k}{\Ctwoalpha{\clo{B_1}}} \le C \norm{g}{\Ctwoalpha{T}} $$
with a constant independent of $k$.
Now consider the Dirichlet problem in the domain $B_1$.
The boundary of this domain consists of the part $\bPart_1 = B_1 \cap T$ and $ \bPart_2 = \partial B_1 \cap \partial B_\delta(x_0)$.
The $u_k \in \Ctwoalpha{\clo{B}}$ from \ref{SolBound} fulfil: 
$$ L u_k = f \ \text{in} \ B_1 $$
$$ u_k = g_k \ \text{on} \ \bPart_1 $$
$$ u_k = u_k \ \text{on} \ \bPart_2 $$
The boundary portion $T_2$ of $B_2$ is compactly contained in $T_1$ therefore by the local estimate for compact boundary portions \ref{EstimateLocalSmoothBoundary} it holds by the uniform bound for the $\Ctwoalphaf$-norms of $g_k$ on $B_1$:
$$ \supr{k}{\norm{u_k}{\Ctwoalpha{B_2 \cup T_2}}} \le C(\norm{u_k}{C^0(\clo{B_1})} + \norm{g_k}{\Ctwoalpha{\clo{B_1}}}+\norm{f}{\Czeroalpha{\clo{B_1}}}) < \infty  $$
Therefore the functions $u_k$ are precompact in $\Ctwoalpha{B_2 \cup T_2}$ and it exists a subsequence converging to $u$ in $\Ctwo{B_2 \cup T_2}$. The $\Ctwof$-limit of a uniformly bounded $\Ctwoalphaf$-sequence is again $\Ctwoalphaf$-smooth (see \ref{LimitPreCompact}) hence $u \in \Ctwoalpha{B_2 \cup T_2}$. Hence the solution is $\Ctwoalphaf$-smooth in $x_0$ which was choosen arbitrarily from $T$.
Thus $u \in \Ctwoalpha{B \cup T}$.
\end{proof}

\begin{definition}{Subharmonic and Superharmonic functions} \label{Sub}
\\Let $L$ be the elliptic differential operator \eqref{defG}, $f \in C^0(\Om)$. A function $u \in C^0(\Om)$ is called subharmonic w.r.t to $(L,f)$ if for every ball $B \subset \subset \Om$, $v$ with $Lv =f \ \text{in} \ B$ and $ u \leq v$ on $\partial B$ it holds $u \leq v$ in $B$. A superharmonic function is defined in the analogous way.  Shortly writing:
\begin{enumerate}
\item Subharmonic: $Lv =f$, $u \leq v$ on $\partial B \Rightarrow u \leq v \ \ \text{in}  \ B$
\item Superharmonic: $Lv =f$, $u \geq v$ on $\partial B \Rightarrow u \geq v \ \ \text{in}  \ B$
\end{enumerate}
\end{definition}

\begin{definition}(Subfunction) { \ } \\
Let $g \in C^0(\Omclo)$, $L$,$f$ as above. A function $u \in C^0(\Omclo)$ is called subfunction to $(L,f,g)$ if $u$ is subharmonic and $u \leq \boundf \ \text{on} \ \partial \Om$.
$S_\boundf$ denotes the set of all subfunctions.
A superfunction is defined analougsly.
\end{definition}

\begin{definition}{Harmonic lifting} \label{HarmonicLifting}
\\Let $\Om$ be a bounded domain, $B=\BR \subset \Om$, $R>0$, $u \in C^0(\clo{\Om})$,$f \in \Czeroalpha{\clo{\Om}}$ , define the \textit{harmonic lifting} $\harm{u}$ to $(L,f,B)$ by
$$ L \harm{u} = f \ \text{in} \ B $$
$$ \harm{u} = u \ \text{on} \ \partial B $$
$$ \harm{u} = u \ \text{in} \ \Om \setminus B$$
By \ref{SolBall} $\harm{u}$ exists and $\harm{u} \in C^0(\clo{B}) \cap \Ctwoalpha{B}$.
Furthermore for every ball $B'$ with $B' \subset \subset B$ we have by the estimate \ref{EllipEstimateComp}:
$$ \norm{\harm{u}}{\Ctwoalpha{B'}} \le C (\norm{\harm{u}}{C^0(B)} + \norm{f}{\Czeroalpha{B}}) $$
\end{definition}

\begin{lemma}{Super- and subsolutions} \label{LemmaSubSuper}
\\Let $\Om$ be a bounded domain, $L$ be the elliptic differential operator \eqref{defG}, $f \in \Czeroalpha{\Om}$, then the following statements hold for super- and subsolutions to $(L,f)$.
\begin{enumerate}
\item A function $u \in \Ctwo{\Om}$ is a subsolution (supersolution) iff $Lu \ge f$ ($Lu \le f$).
\item If $u$ is subsolution and $v$ a supersolution with $u \le v$ on $\partial \Om$ then either $u < v$ or $ u=v$ in $\Om$.
\item If $u$ is a subsolution then for every Ball $B$ with $\clo{B} \subset \Om$ the corresponding harmonic lifting $\harm{u}$ is also a subsolution and $\harm{u} \ge u$.
\item If $u_1$,$u_2$ are subsolutions then max $\{u_1,u_2\}$ is also a subsolution.
\end{enumerate}
\end{lemma}
Remarks on the proof of this lemma can be found in \cite{GilTru} on page 97.

\begin{lemma}{Existence of a super- and subsolution} \label{SubSuperfunction} \\
Let $\Om$ be a bounded domain (contained in the slab $\{0 \leq x_1 \leq d \}$ , $L$ the elliptic differential operator \eqref{defG} with $c\le0$, $g$ a bounded function on the boundary, $f \in C^0(\Omclo)$.
Then a supersolution and subsolution respectively to $(L,f,g)$ is given by:
\begin{enumerate}
\item $v^+ = \supr{x \in \Om}{\mod{g}} + (e^{\gamma d} - e^{\gamma x_1} ) \frac{\norm{f}{C^0(\Omclo)}}{\ellip} $
\item $v^- = - \supr{x\in \Om}{ \mod{g}} - (e^{\gamma d} - e^{\gamma x_1} ) \frac{\norm{f}{C^0(\Omclo)}}{\ellip}$
\end{enumerate}
for a constant $\gamma>0$ depending on the coefficients of $L$ and $d$.
For a general bounded domain $\Om$ the appropriate super- and subsolution is given by a translation. 
\end{lemma}
\begin{proof}
Consider $v^+$. Since the second term is positive for $x \in \Omclo$ it holds $v^+ \ge g \ \text{on} \ \partial \Om$. \label{boundaryg}
Set $f_1 = \frac{\norm{f}{C^0(\Omclo)}}{\ellip}$, then for $\gamma > 0$:

\begin{align*}
L v^+  &= (a_{11} \pdx{1} \pdx{1} v^+ + b_1 \pdx{1} v^+)f_1 + c v^+  \\
&\le -(a_{11} \gamma^2 e^{\gamma x_1} - b_1 \gamma e^{\gamma x_1})f_1 \\
&\le (- \ellip \gamma + \coeffbound) \gamma e^{\gamma x_1} f_1   \\
&= (- \gamma + \frac{\coeffbound}{\ellip}) \gamma e^{\gamma x_1} \norm{f}{C^0(\Omclo)} 
\end{align*}
Choosing $\gamma$ sufficiently large we have $L v^+ \le - \norm{f}{C^0(\Omclo)} \le f$. 
Together with \ref{boundaryg} we have, that $L v^+$ is a supersolution.
The proof for the subsolution works analougsly.
\end{proof}
\begin{remark}
For a solution $u$ with $\LufOm$ and $u = g \ \text{on} \ \partial \Om$ it holds:
$$ v^- \leq u \leq v^+ $$
This prooves the interior sup-norm estimate of a solution \ref{IntMaxEstimate}.
The choice of $\gamma$ determines $\mod{(e^{\gamma d} - e^{\gamma x_1})}$ and therefore the constant $C$ in the interior estimate \ref{IntMaxEstimate}. 
Analyzing the proof one can see, that if $\frac{\coeffbound}{\ellip}$ gets smaller then the lowest possible choice of $\gamma$ decreases and therefore also the constant $C$. This gives the monotone dependence of $C$, as mentioned in \ref{IntMaxEstimate}.
\end{remark}
\begin{theorem}[Interior solution]{\ } \label{IntSol}
\\Let $L$ be the elliptic differential operator \eqref{defG}, $f \in \Czeroalpha{\Omclo}$,$g \in C^0(\Omclo)$. $S_g$ the set of all superfunctions to $(L,f,g)$.
The function 
\begin{eqnarray}u(x) = \supr{v \in S_\boundf} v(x) \label{IntSolEq} \end{eqnarray}
is in $\Ctwoalpha{\Om}$ and fulfills $Lu = f$ in $\Om$. 
\end{theorem}
\begin{proof}
By \ref{LemmaSubSuper}, item (ii) each subfunction is dominated by a superfunction. Hence with the sub- and superfunctions from \ref{SubSuperfunction} we have:
$$ v^- (x) \le u(x) \le v^+ (x), \ x \in \Om $$
thus $ \mod{u(x)} < \infty$. \newline
Let $x_0 \in \interior{\Om}$. Choose two Balls $B'(x_0)$, $B(x_0)$ with  
$$B'(x_0) \subset B(x_0) \subset \subset \Om$$
By definition of the sup there exists a sequence of subfunctions $v_k \ \in C^0(\Omclo)$  such that $v_k(x_0) \limk u(x)$. Define $v'_k = \harm{max \{v_k,\infi{x \in \Om}{u(x)}\}}$, the harmonic lifting to $(L,f,B)$. Then $v'_k \in \Ctwoalpha{B}$. 
Since $$\mod{max \{v_k,\infi{x \in \Om}{u(x)}\}} \le \mod{u}$$ we have by the interior maximum estimate \ref{IntMaxEstimate}: 
$$ \norm{v'_k}{C^0(\clo{B})} \le (\norm{u}{C^0(\Omclo)} + C \norm{f}{C^0(\Omclo)})$$
Thus the $v'_k$ are  bounded. \newline
Furthermore $v'_k \geq v_k$ in $B$, thus $v'_k(x_0) \limk u(x_0)$. 
By \ref{HarmonicLifting} we have:
\begin{multline} \norm{v'_k}{\Ctwoalpha{B'}} \le C (\norm{v'_k}{C^0(\clo{B})} + \norm{f}{\Czeroalpha{\clo{B}}}) \le C_2( \norm{u}{C^0(\Omclo)} + \norm{f}{\Czeroalpha{\Omclo}})
\end{multline}
Thus the functions $(v'_k)_k$ are uniformly bounded in $\Ctwoalpha{B'}$. By the compactness of the embedding $ \Ctwoalpha{\clo{B'}} \emb \Ctwo{\clo{B'}}$ there is a subsequence converging in $\Ctwof$ to some $v'$. This $v'$ satisfies $L v' = f$ in $B'$ and $u(x_0) = v'(x_0)$. Since the sequence converges in $\Ctwof$-norm and is uniformly bounded in $\Ctwoalphaf$ the limit $v'$ is also in $\Ctwoalphaf(B')$ (by \ref{LimitPreCompact}). \newline
Assume now there exists a $x_1 \in B'$ such that: $v'(x_1) < u(x_1)$. Let $(w_k) \in S_g$ such that $w_k(x_1) \limk u(x_1)$. Define $w'_k = \harm {max \{v',w_k\}}$ again by compactness we find a $\Ctwoalphaf$-smooth $w$ with $Lw=f$ and $w \geq v'$ in $B$. Hence $L(v'-w) = 0 \ \ \text{in}  \ B'$, $(v'-w) \leq 0$. But at $x_0$ we have $u(x_0) = v'(x_0) =w(x_0)$. By the strong maximum principle \ref{StrongMax} the function $v-w$ must be constant, contradicting the assumption. Hence $u(x) = v'(x)$ in $B'$ and therefore $u(x)$ is $\Ctwoalphaf$-smooth in $x_0$ and satisfies $(L u)(x_0) = f(x_0)$. 
\end{proof}
The theorem above provides an interior solution for \ref{DirProblemEllip}. The next step is now to show that under certain conditions on the boundary, the solution attains the boundary values $\boundf$ continuously. These conditions are especially fulfilled for $C^2$-smooth boundaries and cuboid domains.

\newcommand{\wi}{w^{+/-}_\epsi}
\newcommand{\wip}{w^+_\epsi}
\newcommand{\wim}{w^-_\epsi}
\begin{definition}{Barrier} \label{Barrier}
\\Let $(L,f,g)$ as in \ref{IntSol}. A net of functions $(\wi)_{\epsi>0} \in \Ctwo{\Om} \cap C^0(\Omclo)$ is called a barrier to $(L,f,g)$ for a point $x_0$, if:
\begin{enumerate}
\item $L\wip \leq f$, $L\wim \geq f \ \ \text{in}  \ \Om$
\item $\wip \geq \boundf$, $\wim \leq \boundf \ on \ \partial \Om$
\item $\wi(x_0) \overset{\epsi \rightarrow 0}{\longrightarrow} \boundf(x_0)$  
\end{enumerate}
 
\end{definition}
Each $\wip$ is a superfunction and each $\wim$ is a subfunction. Hence $\wim(x) \leq u(x) \leq \wip(x) $ therefore we get the following lemma:

\begin{lemma}{Continous attaining of boundary values} \label{ContBoundary}
\\Let $(L,f,g)$ as in \ref{IntSol}. 
If $x_0$ has a barrier then for $u(x) = \supr{v \in S_\boundf}$, $g \in C^0(\Omclo)$ it holds $ u(x) \rightarrow u(x_0)$ for $ x \rightarrow x_0$, i.e. $u(x) \in C^0(\Omclo)$.
\end{lemma}
\begin{proof}
Since $w^+_\epsi$ and $w^-_\epsi$ are superfunctions and subfunctions respectively we have for all $\epsi>0$:
$ w^-_\epsi(x) \le u(x) \le w^+_\epsi(x) \ \text{for } x \in \Om $.
Since $w_\epsi(x_0) \overset{\epsilon \rightarrow 0}{\longrightarrow}  g(x_0)$ and $w_\epsi$ are continuous we have 
$u(x) \rightarrow g(x_0)$ if $x \rightarrow x_0$.
\end{proof}

\begin{lemma}{External sphere condition} \label{ExtSphere} \\
Let $L$ be the elliptic differential operator \eqref{defG}, $g \in C^0(\Omclo)$, $f \in \Czeroalpha{\Om}$. 
A point $x_0$ on $\partial \Om$ satisfies the external sphere condition, if there exists a sphere $S$ such that $S \cap \partial \Om = \{x_0 \}$. In this case the interior solution $u (x) = \supr{v \in S_g}{v(x)} $ is continuous at this point. 
In particular, if all points satisfy the external sphere condition then the solution is in $\Ctwoalpha{\Om} \cap C^0(\Omclo)$.
\end{lemma}
\begin{proof}
Let $x_0 \in \partial \Om$ and $B_R(y)$ such that: $B_R(y) \cap \partial \Om = \{x_0\}$. By translation we may assume $y=0$. Define $w(x) = \tau(R^{-\sigma}-r^{-\sigma}), \ r=|x|$. For $x \in \Om$: 
\begin{alignat*}{1}
 L w (x) =\tau L (R^{-\sigma}-r^{-\sigma}) 
&= \tau \sigma r^{-(\sigma+4)}(-(\sigma+2) \sumij{a_{ij}x_ix_j} + r^2(\sumi{a_{ii}+b_ix_i}) ) + \tau c (R^{-\sigma}-r^{-\sigma}) \\
 &\leq \tau \sigma r^{-(\sigma+2)}(-(\sigma+2)\ellip + \sumi{a_{ii}+b_ix_i}) 
\end{alignat*}
The last equality follows since $ c (R^{-\sigma}-r^{-\sigma}) \le 0 $ and $a_{ij}x_ix_j > \ellip \norm{x}{}^2=\ellip r^2$. \\
Since $\Om$ is bounded, we have $\sumi{a_{ii}+b_ix_i} = C_1 < \infty$. Thus for $\sigma$ large enough the second factor is strictly negative, then choose $\tau$ large enough such that
$$ L w(x) \leq -1 \ \text{in} \ \Om $$
Furthermore $w(x_0) = 0$ and by the external sphere condition $dist(\Om,0) > \mod{x_0} = r$, \\
$w(x) > 0 \ \text{in} \ \Om \setminus {x_0}$. 
Using this function we can now construct a net of barriers. \\
Let $\epsi > 0$.
Since $g$ is continuous there exists a neighborhood $U$ with $g(x) < g(x_0)+ \epsi$. \\
Since $w(x) > \delta > 0$ for all $x \in \partial \Om \setminus (\partial \Om \cap U)$, 
there exists a $k_\epsi$ such that: $$w^+_\epsi(x) := \boundf(x_0)+\epsi+k_\epsi w(x) \geq \boundf(x)$$

Substitute $k_\epsi' = max(k_\epsi, \supr{\Om}{\mod{f}+\coeffbound \mod{\boundf(x_0)}})$ then

\begin{alignat*}{1} 
Lw^+_\epsi(x) &= c(\boundf(x_0) +\epsi) + k_\epsi L w(x) \notag \\
&\leq c(x) \boundf(x_0) - k'_\epsi \notag \\
&\le c(x) \boundf(x_0) - \mod{f(x)} - \coeffbound \mod{\boundf(x_0)} \notag \\
&\le - \mod{f(x)} \notag \\
\end{alignat*} 

Hence $w^+_\epsi$ defines an upper barrier analougsly $w^-_\epsi = \boundf(x_0) - \epsi -k_\epsi w(x)$.  

By \ref{ContBoundary} the solution is continuous in $x_0$. 

\end{proof}

\begin{theorem}[Smooth solution on smooth boundary portion]{ \ } \label{SmoothBoundaryPortion}
\\If $T$ is a $\Ctwoalphaf$-smooth boundary portion and $\boundf \in \Ctwoalpha{\Om \cup T}$ then the solution $u$ \eqref{IntSolEq} is in 
$\Ctwoalpha{\Om \cup T}$. If $\partial \Om$ is $\Ctwoalphaf$ then the solution is in
$\Ctwoalpha{\Omclo}$.
\end{theorem}
\begin{proof}
See \cite{GilTru}, Theorem 6.14, page 101.
\end{proof}
Summarizing we have seen that for $\Ctwoalphaf$-smooth boundaries and $f \in \Czeroalpha{\Omclo}$ the Dirichlet problem \ref{DirProblemEllip} has a $\Ctwoalpha{\Omclo}$-smooth solution. In cuboid domains we have smoothness of the solutions except for the boundary corners and edges, the solution is however continuous on the whole boundary.

\section{Elliptic operators with $\Ctwoalphaf$-coefficients in $C^{4,\alpha}$ smooth domains}
\begin{explainVor}
In this section we consider the assumptions \ref{assSemi}, \ref{assCont} on domains having $C^{4,\alpha}$ smooth boundary. For the differential operator \eqref{defG} assume additionally that the coefficients $a_{ij},b_{i} \in \Ctwoalpha{\Omclo} \ \text{for} \ i,j \ \le n$.
The additional assumptions on the coefficients are made to construct the contractive extension operator, the semigroup exists also for coefficients in $\Czeroalphaf$. The higher smoothness assumptions of the domain are made to ensure that also the transformed differential operator (see \ref{TrafoDiffOp}) has $C^{2,\alpha}$ smooth coefficients.
Recall that the matrices $a_{ij}$ are assumed to be uniformly elliptic with ellipticity constant $\ellip>0$ and the coefficients are bounded by $\coeffbound$. 
Throughout this chapter $\Om$ is a bounded domain in $\R^n$, $x_n$ denotes the n-th coordinate of a point in $\R^n$.
\end{explainVor}

\begin{lemma}[Density of $C^{0,\alpha}_0(\Omclo)$ in $C^0_0(\Omclo)$]
The space $C^{0,\alpha}_0(\Omclo)$ is dense in $C^0_0(\Omclo)$.
\end{lemma}
\begin{proof}
We have the inclusion $C^\infty_c (\Omclo) \subset C^{0,\alpha}_0 (\Omclo) \subset C^0_0(\Omclo)$. 
By a generalisation of the Stone-Weierstrass-Theorem $C^\infty_c(\Omclo)$ is dense in $C^0_0(\Omclo)$, see \ref{StoneWeierZero}
\end{proof}

\begin{theorem}[Operator semigroup] { \ } \label{EllipOpSemigroup} \\
Let $\Om$ be a bounded domain with $\Ctwoalphaf$-smooth boundary. $(L,D(L))$ the differential operator \eqref{defG}, $D(L)$ as in \eqref{defDL}.
Then the closure of the operator $(L,D(L))$ generates a strongly continuous semigroup on $C^0_0(\Omclo)$.

\end{theorem}
\begin{proof}
Set $\omega := \supr{x \in \Om}{\ c(x)}$. By \ref{Dissipative} the operator $L - \omega$ is dissipative. Now consider for $\lambda > \omega$, $f \in \Czeroalpha{\Omclo}$ the resolvent equation multiplied with $-1$:
\begin{eqnarray}
 (L - \lambda) u = f \ \text{on} \ \Om \\
 u = 0 \ \text{on} \ \partial \Om \nonumber
\end{eqnarray}
This defines an elliptic equation as in \eqref{DirProblem} with  $\tilde{c} = c - \lambda \le 0$.
By \ref{IntSol} there exists a unique solution in $C^0_0(\Omclo) \cap \Ctwoalpha{\Om}$. By the smoothness of the boundary values and the boundary we have with \ref{SmoothBoundaryPortion} that $u \in \Ctwoalpha{\Omclo}$.
Hence the image of  $D(L) = \DL$ under $\lambda -L$ contains $\Czeroalpha{\Omclo}$ and is therefore dense in $C^0_0(\Omclo)$. 
By \ref{TheoDissipOmega} the closure generates a strongly continuous semigroup.
\end{proof}
\begin{explainText}
Thus assumption \ref{assSemi} is proved. Now we construct the continuation operator.
The operator is constructed locally and uses the symmetries of the differential operator. First the boundary is straighten out, which transforms also the differential operator. Then the symmetries of the resulting operator together with the boundary conditions are used to continue the function by a squeezed reflection in a certain direction. 
\end{explainText}
\begin{explainVor}
The standard continuation operators, see e.g. \cite{GilTru}, Lemma 6.3, p. 131  are not suitable, since they are in general not contractive. They continue a function outside a domain by a weighted sum of the function values inside the domain. These weights have different signs and the sup-norm of the continued function is in general strictly larger. The extension operator presented here overcomes this difficulty by doing a certain reflection, i.e. each function value outside corresponds to a function value inside, so the sup-norm stays the same.
To guarantee the smoothness of the continued function it is however necessary to restrict this operator to $\Ctwoalphaf$-functions with the additional boundary conditions $u=0$, $Lu=0$ as in \eqref{defG}.
Therefore the extension operator and the generator of the semigroup are closely related.
The construction of the extension operator starts locally. In order to define the extension the boundary has to be flatten out and the differential operator has to be transformed in a suitable form. For this we need the notation of a transformed differential operator (\ref{TrafoDiffOp}). The existence of a suitable transformation is given in Lemma (\ref{TrafoNoMix}). Then in Theorem (\ref{ExtenStraight}) the local extension operator is constructed. Finally in Theorem (\ref{Exten}) the global extension operator is constructed.
\end{explainVor}


\newcommand{\utwo}{u_2 \circ \floone}

%
\begin{definition}[The Pullbackoperator and the transformation of a differential operator] \label{TrafoDiffOp}
Let $\Om_1$, $\Om_2 \subset \R^n$ be two domains. $F : \Om_1 \rightarrow \Om_2$ a $C^{4,\alpha}$-diffeomorphism. This diffeomorphism induces a continuous isomorphism the \textit{Pullbackoperator}
\begin{eqnarray}T_F : C^{s,\sigma}(\Om_2) \rightarrow C^{s,\sigma}(\Om_1), \ \text{for} \ s+\sigma \leq 4+\alpha \label{pullbackOp} \\
T_F(g) = g \circ F \ , \ g \in C^{s,\sigma}(\Om_2) \nonumber
\end{eqnarray}

Let $L_1 : \Ctwoalpha{\Om_1} \rightarrow \Czeroalpha{\Om_1}$ be a differential operator of second order as in (\ref{defG}) then we define the $\textit{transformed differential operator}$ $L_2 : \Ctwoalpha{\Om_2} \rightarrow \Czeroalpha{\Om_2}$ by $ L_2 = (T_F)^{-1} \circ (L_1 \circ T_F) : \Ctwoalpha{\Om_2} \rightarrow \Czeroalpha{\Om_2}$. 
For $u_2 \in \Ctwoalpha{\Om_2}$ with $u_1 = u_2 \circ F$, $y_0 = F(x_0)$, $x_0 \in \Om_1$ we have:
$ (L_2 u_2)(y_0) = (L_1 u_1)(x_0)$. Using the chain rule, the derivatives up to second order w.r.t. to the original variable $x$ can be written in terms of the derivatives w.r.t. to the new variables $y$. For $i,j \in \{ 1,...,n \}$ we have:
$$ \pdx{i} \pdx{j} u_1 = \sum^{n}_{k,l=1} (\pdy{k} \pdy{l} u_2) (\pdx{i} F_k)(\pdx{j} F_l) + \pdx{i} \pdx{j} F_k \pdy{k} u_2. $$
So the drift coefficients of the transformed differential operator contain derivatives of second order of the diffeomorphism. If $F \in C^{4,\alpha}(\Om_1)$ then $\tilde{b}_i \in C^{2,\alpha}(\Om_2)$.
(See also \cite{GilTru} on page 91).

Hence there exist $\mtilde{a_ {ij}}$,$\mtilde{b_{i}} \in C^{2,\alpha}(\Om_2)$, $\mtilde{c} \in C^{0,\alpha}(\Om_2)$ such that:

$$ L_2 u_2 = \sumij \mtilde{a_{ij}} \pdy{i} \pdy{j} u_2 + \sumi \mtilde{b_i} \pdy{i} u_2 + \mtilde{c} u_2. $$ \label{trafoL2}
\end{definition}

\begin{diagram}
 \Om_1 & \rTo{F} & \Om_2 \\
 \Ctwoalpha{\Om_1} & \lTo{T_F} & \Ctwoalpha{\Om_2} \\
 \dTo^{L_1} &  & \dTo_{=: L_2} \\
 \Czeroalpha{\Om_1} & \rTo{T^{-1}_F} & \Czeroalpha{\Om_2} \\
\end{diagram}

\begin{diagram}
 x_0 & \rTo{F} & y_0 \\
 (u \circ F)(x_0) & \lTo{T_F} & u(y_0) \\
 \dTo^{L_1} &  & \dTo_{=: L_2} \\
 (L_1 u \circ F)(x_0) & \rTo{T^{-1}_F} & (L_2 u)(y_0) \\
\end{diagram}
\newcommand{\yz}{y^{(0)}}
\newcommand{\xz}{x^{(0)}}
\newcommand{\ainx}{a_{i,n}(x_0)}
\newcommand{\tildeainy}{\mtilde{a_{i,n}(y_0)}}
\newcommand{\annx}{a_{n,n}(x_0)}
\newcommand{\tildeanny}{\mtilde{a_{n,n}(y_0)}}

\begin{lemma}{Transformation to an operator with no mixed derivatives} \label{TrafoNoMix}
\\Let $R > 0$, $\Om_1 = B_R(0) \subset \R^n$, $L_1$ the differential operator \eqref{defG}) acting on  $C^{2,\alpha}(\BRplus)$ with coefficients 
$$a_{ij}, b_i \in C^{2,\alpha}(B^{0,+}_R(0)) \ \text{for} \ i,j \le n$$
For a point $x \in \R^n$ denote by $x^{(n-1)}$ the vector of the first (n-1) components and by $x_n$ the n-th component.
Then there exists a $\Ctwoalphaf$-smooth map $F : \Om_1  \rightarrow F(\Om_1)$
with the following properties: 
\begin{enumerate}
\item For $x \in \Om_1$ with $x_n > 0$ it holds $(F(x))_n > 0$
\item For $x \in \Om_1$ with $x_n = 0$ it holds $(F(x))_n = 0$
\item There exists a $0<R'\le R$ such that the restriction $\floone = F \lvert_{B_{R'}(0)}$ is an $C^{2,\alpha}$-isomorphism, i.e. $\floone : B_{R'}(0) \isomorph  \floone(B_{R'}(0))$  
\item The transformed differential operator $L_2 = T^{-1}_{\floone} \circ  L_1 \circ T_{\floone}$, has no crossterms of second order for points in $\BRzero$. That is for $y = F(x)$ with $x_n=0$ it holds $y_n =0$ and
$\mtilde{a_{in}} (y) = 0$ for $1 \le i < n$, where $\mtilde{a_{in}}$ denote the crossterms of $L_2$.
\end{enumerate}
\end{lemma}

\begin{proof}

Denote by $\firstnx$ the first $(n-1)$ coordinates of a point $x$.
For $1\le i < n$ define
$$g_i : \Om_1 \rightarrow \R , \ g_i(x) = g_i(\firstnx) = - \frac{a_{in}(\firstnx,0)}{a_{nn}(\firstnx,0)}$$
$g_i$ is a $\Ctwoalpha{\clo{\Om_1}}$-function because the coefficients are assumed to be in $\Ctwoalpha{\clo{\Om^+_1}}$ and since $L_1$ is uniformly elliptic $a_{nn} > 0$. Define $g : \Om_1 \rightarrow \R^{n-1}$ by $(g(x))_i =g_i(x)$. Define

$$ F : \Om_1 \rightarrow \R^n ,\ F(x) = (\firstn + g(x)x_n,x_n) $$
Denote by $\Om_2 = F(\Om_1)$. 
$$ (D F)(x) = 
\Big{(} \begin{array}{ll}
Id +(D g(x))x_n & g(x) \\ 
0 & 1
                        \end{array} \Big{)}  \ \text{for all } x \ \in \Om_1
$$
Now for $x=0$, it holds $ Det (D \floone)(0) = 1$.
By the Inverse-Function theorem \ref{TheoDiff} $F$ is a local $\Ctwoalphaf$-diffeomorphism in a neighborhood of $0$.  
Hence there exists a $R'>0$ with $R' \le R$ such that $\floone:=F \lvert_{B_{R'}(0)}$ is a diffeomorphism. Now substitute $R$ by $R'$ and denote by $\Om_1 = \BR$,$\Om_2=\floone(\Om_1)$.
In order to calculate the coefficients of the transformed differential operator $L_2$ we have to write the partial derivatives in the original variable $x$ in term of derivatives in the new variable $y$.
By the chain rule we have $\nabla_x =\nabla_y (D\floone)$ in the sense that $(\nabla_x u_1)(x_0) =(\nabla_y u_2)(y_0) (D\floone)(x_0)$ for $u_2 \in \Ctwo{\Om_2}$, $u_1 = u_2 \circ \floone$. For convenience we use this notation in the following, i.e. we identify the transformed differential operator and the differential operator itself. 

\begin{scherz}
 \textit{Dafür brauchen sie nur einen Bleistift. Einen spitzen Bleitstift.} (Freeden)
\end{scherz}

\begin{align*}
 \pdx{i} &= \pdy{i} + x_n \sumfromtoj{1}{n-1} (Dg)_{ji} \pdy{j}  \\
 \pdx{n} &= \pdy{n} + \sumfromtoj{1}{n-1} g_j(\firstn) \pdy{j}  \\
\pdx{i}\pdx{j}&=(\pdy{i} + x_n \sumfromtoj{1}{n-1} (Dg)_{ji} \pdy{j}  )(\pdy{j} + x_n \mysum{k}{n-1} (Dg)_{kj} \pdy{k}) \\
\pdx{i}\pdx{n} &= (\pdy{i} + x_n \sumfromtoj{1}{n-1} (Dg)_{ji} \pdy{j} )(\pdy{n}+\sumfromtoj{1}{n-1} g_j(\firstn) \pdy{j}) \\
& = \pdy{i}\pdy{n} + x_n(\sumfromtoj{1}{n-1} (Dg)_{ji} \pdy{j}\pdy{n}  + P_{(n-1)}) \\
\pdx{n}\pdx{n} &= (\pdy{n} +g^\top( \nabla_y^{(n-1)} )^\top )(\pdy{n}+g^\top( \nabla_y^{(n-1)} )^\top) \\
&= \pdy{n}\pdy{n} + 2g^\top(\pdy{n} \nabla_y^{(n-1)} )^\top + P_{n-1}
\end{align*}
where $P_{(n-1)}$ denotes (different) terms only depending on partial derivatives of first or second order w.r.t to the coordinates $1$ to $n-1$.
For $x_0$ with $x_n = 0$ we have: 
\begin{align*} \pdx{i}\pdx{n} &=  \pdy{i}\pdy{n} + P_{(n-1)} \\
 \pdx{n}\pdx{n} &= \pdy{n}\pdy{n} + 2 \sumfromto{1}{n-1} g_i(x_0) \pdy{i}\pdy{n} \\
\end{align*}
Plugging this in the definition of $L_1$:
\begin{align*} L_1 &= P_{(n-1)} +  2 a_{in} \pdx{i}\pdx{n} + a_{nn} \pdx{n}\pdx{n} & \notag \\ 
&= P_{(n-1)} + 2 a_{in} (\pdy{i}\pdy{n} + P_{(n-1)}) + a_{nn} (\pdy{n}\pdy{n} + 2 \sumfromto{1}{n-1} g_i(x_0) \pdy{i}\pdy{n}) & \notag \\
&= P_{(n-1)} + 2 (a_{in} + a_{nn} g_i(x_0)) \pdy{i}\pdy{n} + a_{nn} \pdy{n}\pdy{n} & \notag 
\end{align*}
By the definition of $g$ we have for every $x_0 \in \BRzero$:
\begin{align*}
a_{in}(x_0) + a_{nn} g_i(x_0) &= (a_{in} +a_{nn}(- \frac{a_{in}}{a_{nn}}))(x_0)=0 \\
\intertext{Thus:}
\mtilde{a_{in}}(y_0) &= a_{in}(x_0) + g_i(x_0) a_{nn}(x_0) = 0 \\
\end{align*}

\end{proof}

\newcommand{\uextend}{- u (\firstn,-(x_n+ex_n^2))}
\newcommand{\Flotwo}{F^2}
\newcommand{\exfunc}{-(x_n+ex_n^2)}

\newcommand{\bezDL}{S}

\newcommand{\Flotwoab}{F_{a,b}^{(2)}}

\begin{lemma}[Construction of the reflection function] \label{ReflFunc}
Let $R>0$, $a,b \in \R$, $a > 0$. Then there exists a $\delta>0$ and a twice differentiable function $\Flotwoab :\R \rightarrow \R $ which maps the interval $(-\delta,0]$ to $[0,R)$ and 
\begin{eqnarray} 
 \Flotwoab(0)=0 \nonumber \\
 (\Flotwoab)^{'} (0) = -1 \label{derivative} \\
 (\Flotwoab)^{''} (0) = 2 \frac{b}{a} \nonumber 
\end{eqnarray}
Furthermore the largest possible $\delta$ depends montone increasing on $a$ and decreasing on $b$.
Denote by $F : \R^+ \times \R \times \R \rightarrow \R$ the function $F(a,b,x) =\Flotwoab(x)$.
Then for parameters $a$,$b$ from subsets $A=[\ellip,\infty)$, $B=[-\coeffbound,\coeffbound]$, $\ellip>0$,$0<\coeffbound<\infty$ there exists a common $\delta>0$ such that
$F( A \times B \times (-\delta,0]) \subset [0,R)$.
And the function $F(a,b,x) = F_{a,b} (x)$ depends $C^\infty$-smooth on $a$,$b$ for $a>0$.
\end{lemma}
\begin{proof}
For $s \in \R$ define $$\Flotwoab(s) := - (s - \frac{b}{a} s^2)$$
If $b \neq 0$ then $\Flotwoab(s) = \frac{b}{a} s (s - \frac{a}{b})$, if $b=0$ then $\Flotwoab(s) = -s$. \newline
Obviously \ref{derivative} is fulfilled and for $a \ne 0$ the function is $C^\infty$-smooth in the coefficients.
If $b=0$ then $\Flotwoab$ is a linear function and maps $(-R,0]$ to $[0,R)$. 
If $b<0$ then the function is positive between the negative zero $\frac{a}{b}$ and $0$. 
If $b>0$ then the function is positive for all negative $s$. 
The choice of $\delta$ depends therefore on the negative zero of $\Flotwoab(s)$ if $b<0$ and on the negative zero of $\Flotwoab(s) - R$ if $b>0$.
The second zero of $\Flotwoab$ is given by $s_2=\frac{a}{b}$.
If $b<0$ then $\Flotwoab(s) \geq 0$ for $s \in [s_2,0]$ and $\Flotwoab(s) < -s$ for $s<0$.
\\ Hence for $\delta < \rm{min} \{| \frac{a}{b} |,R\}$ it holds $\Flotwoab((-\delta,0]) \subset [0,R)$.

If $b\geq0$ then $\Flotwoab(s) >0$ for all $s < 0$. Let $s_R$ be the negative zero of the function $G(a,b,s,R) = F(a,b,s) -R$, then for $\delta < |s_R|$ it holds $\Flotwoab((-\delta,0]) \subset [0,R)$. Using the implicit function theorem one can see, that the zero $s_R = s_R(a,b)$ depends differentiable on $a$ and $b$ and that $\partial_a s_R < 0$ and $\partial_b s_R > 0$ for $a \in A$,$b \in B$. 
Thus $|s_2|$ gets larger if $a$ gets larger or if $b$ gets smaller. If the negative zero gets smaller than the possible choice for $\delta$ gets larger. Thus, for $a \in [\ellip,\infty)$, $b \in [-\coeffbound,\coeffbound]$ it holds $\delta(a,b) \geq \delta(\ellip,\coeffbound) =: \delta_0$ and thus $$F (\{a\} \times \{b\} \times  (-\delta_0,0])  \subset F (\{a\} \times \{b\} \times (-\delta(a,b),0] )\subset [0,R)$$ and $$F (A \times B \times (-\delta,0]) \subset [0,R)$$

 \end{proof}

\newcommand{\pds}{\partial_s}
\newcommand{\pdsplus}{\partial^+_s}
\newcommand{\pdsminus}{\partial^-_s}

\begin{lemma} [Extension operator in 1 dimension] \label{Exten1dim}
 Let $a >0$, $b \in \R$. Then there exists a continuous linear extension operator $E$ from the Banach space 
$$X_1 := \{ u \in C_c^{2,\alpha}(\R^+_0) \ | \ a \pds \pds u(0) + b \pds u(0) = u(0) = 0 \ , \supp(u) \subset [0,C) \text{for some} \ C > 0 \} $$ \label{Labrel}
into $\Ctwoalphac{\R}$ such that $\norm{Eu}{C_c^0(\R)} = \norm{u}{C_c^0(R^+)}$ and $\text{supp}(Eu) \subset [-\delta,\infty)$ for some $\delta>0$.
\end{lemma}

\begin{proof}
 Let $F = \Flotwoab$, $\delta = \delta(a,b)$ as in \ref{ReflFunc}.  Define 
$$
(Eu)(s) = 
\begin{cases}
 - u(F(s)) & - \delta < s < 0  \\
 u(s) & s \geq 0
\end{cases}
$$ \newline
Since $F$ maps $(-\delta,0)$ into $\R^+$ the operator is welldefined. Furthermore for $s>0$ or $s<0$ the function $Eu$ is a composition of $\Ctwoalphaf$-smooth functions and therefore itself $\Ctwoalphaf$-smooth for $s \ne 0$. It is left to show $\Ctwoalphaf$-smoothness for $s=0$. 
Since $u \in \Ctwoalpha{\R^+_0}$ we have by \ref{DiffBoundary} that the one sided derivative $\pds^+ u (0)$,$\pds^+ \pds^+ u (0)$ exist and coincide with the corresponding continuous continuation of the derivative in the interior. Hence the boundary conditions imply:
\begin{align}
 (a \pdsplus \pdsplus u + b \pdsplus u)(0) &= 0  \label{DerPdsplus} \\
\intertext{For $s_0<0$: }
(\pds Eu) (s_0) &= \pds u (F(s_0))(F'(s_0)) \notag \\
(\pds \pds Eu) (s_0) &= \pds \pds u(F(s_0))(F'(s_0))^2 + \pds u (F(s_0))(2 \frac{b}{a}) \notag 
\end{align}
Note that the function $F$ is monoton decreasing near $0$ thus a lower derivative turns into an upper derivative: 
$$\pdsminus (Eu) (0) =- \pdsminus (u\circ F) (0) =- (\pdsplus u(F(0))) F'(0) \overset{\ast}{=} \pdsplus  u(0) = \pdsplus(Eu)(0)$$
the equality $\ast$ follows by $F(0) =0$, $F'(0) = -1$.

\begin{align} \pdsminus \pds (Eu) (0) &= - \pdsminus (\pdsminus (Eu))(0) \\
&= -\pdsminus (((\pdsplus u) \circ F)F')(0) \notag \\
&= -(\pdsplus \pdsplus u)(F')^2 (0) - ((\pdsplus u))(2\frac{b}{a}) (0) \notag \\
&= - \pdsplus \pdsplus u(0) - \pdsplus u (2\frac{b}{a})(0) \notag
\intertext{By \ref{DerPdsplus} we have $ b \pdsplus u (0) = -a \pdsplus \pdsplus u (0)$ and thus}
&= (- \pdsplus \pdsplus u + 2 \pdsplus \pdsplus u) (0) = \pdsplus \pdsplus u (0) \notag
\end{align}
Hence $\pdsminus (Eu) (0) = \pdsplus (Eu) (0)$ and $\pdsminus \pdsminus (Eu) (0) = \pdsplus \pdsplus (Eu) (0)$. The second derivative is also in $\Czeroalpha{(-\delta,\infty)}$, since it is Hölder continuous in both subintervals $(-\delta,0]$ and $[0,\infty)$.
Thus $Eu$ defines a function in $\Ctwoalpha{(-\delta,0] \cup \R^+)}$. Choose now a cutoff $\eta$ for $\R^+$ and $(-\delta,0]$ and define finally:
$$ E_1 u = \eta Eu $$ 
then $E_1 : X_1 \rightarrow \Ctwoalphac{\R}$ and obviously $\norm{u}{C_c^0(\R)} = \norm{u}{C_c^0(\R^+)}$.
$E_1$ is a linear operator from the Banach space $(X_1,\norm{\cdot}{\Ctwoalphac{\R_0^+}})$ to $(\Ctwoalphac{\R},\norm{\cdot}{\Ctwoalphac{\R}})$ and continuous from $(X_1,\norm{\cdot}{C^0_0(\R)})$ to $(\Ctwoalphac{\R},\norm{\cdot}{C^0_c(\R)})$. By \ref{ContLinOp} $E_1$ is therefore also continuous from $(X_1,\norm{\cdot}{\Ctwoalphac{\R_0^+}})$ to $(\Ctwoalphac{\R},\norm{\cdot}{\Ctwoalphac{\R}})$.  
\end{proof}

\begin{remark} \label{remCont}
The $\Ctwoalphaf$-norm of $E_1 u$ can be explicitly estimated by the $\Ctwoalphaf$-norm of $u$, the reflecting function $\Flotwoab$ and the cutoff $\eta$.
\end{remark}

\def\Qn{Q^{(n-1)}}
\begin{theorem} [Extension operator on straight boundary] { \ } \label{ExtenStraight} \\
Let $R>0$, $Q = [-R,R]^n$, $Q^{0,+} = [-R,R]^{n-1} \times [0,R]$.
Let $L$ be the elliptic differential operator \eqref{defG} with coefficients $a_{ij}$,$b_i \in C^{2,\alpha}(Q^{0,+})$. Assume furthermore, that $L$ has no crossterms, i.e. for $x \in Q$ with $x_n=0$ it holds  $a_{in}(x) = 0$ for $i \le n$.  Denote by $\ellip$ and $\coeffbound$ the ellipticity constant and the bound of the coefficients respectively. Then there exists a linear (continuous) extension operator $E$ from the space 
$$ X_1 := \{ u \ \in C^{2,\alpha}(Q^{0,+}) \ | \ u(x) = (Lu)(x) = 0 \ \text{for all} \ x \ \text{with}  \ x_n=0  \} $$
$$ E: X_1  \longrightarrow \Ctwoalpha{Q} $$
such that $\norm{Eu}{C^0(Q)} = \norm{u}{C^0(Q^{0,+})}$
\end{theorem}
\begin{proof}
Denote by $\Qn = [-R,R]^{n-1}$.
For $x =(\firstnx,x_n) \in Q$ define $$F(x) = F(a_{nn}(\firstnx,0),b_n(\firstnx,0),x_n)$$
where $F$ is the reflecting function from \ref{ReflFunc}. 
Set $\delta := \delta(\ellip,\coeffbound)$. Then $\delta > 0$ and $\delta \leq \delta(a_{nn}(x),b_{n}(x))$ with $\delta(a,b)$ as in \ref{ReflFunc}. Therefore $F(x)$ maps elements from $\Qn \times (-\delta,0]$ into $\Qn \times [0,R)$.
Define 
$$
(Eu)(x) = 
\begin{cases}
- u(\firstnx,F(x)) & - \delta < x_n < 0  \\
 u(x) & x_n \geq 0
\end{cases}
$$
As in the 1-dimensional case $\norm{Eu}{C^0(Q)} = \norm{u}{C^0(Q^{0,+})}$.
For $x_n > 0$ this function is $\Ctwoalphaf$. For $x_n<0$ $(Eu) = u \circ F \circ (a_{nn} (x),b_{n}(x),x)$. Since $F$ is $C^\infty$-smooth and the coefficients are $\Ctwoalphaf$ $(Eu)$ is also $\Ctwoalphaf$.
By \ref{DiffBoundary} the partial derivatives $(\pdx{i} u)(x_0)$, $(\pdx{i} \pdx{j} u)(x_0)$ for $x_0 \in Q_0$ and $i,j < n$ exist and furthermore the one sided limits $(\pdx{n}^+ u) (x_0)$, $(\pdx{n}^+ u)(x_0)$, $(\pdx{n}^+ \pdx{i} u)(x_0)$,$(\pdx{i} \pdx{n}^+ u)(x_0)$ exist.
Since for $x_n=0$ we have $(Eu)(x_0) = u(x_0)$, $Eu$ is twice continuously differentiable in directions $1$ to $(n-1)$. By \ref{Exten1dim} the function $Eu$ is twice continuously differentiable in $x_n$. 
Thus for $x_0 \in Q^0$ $Eu$ is differentiable, twice differentiable w.r.t $x_n$ and twice differentiable in the variables $x_i$ for $i <n$. It is left to show the existence of the mixed derivatives $\pdx{i}\pdx{n}$ and $\pdx{n}\pdx{i}$.
Denote by $u_{in}$ ($i < n$) the continuous extension of $\pdx{i}\pdx{n} u$ to $Q^{0,+}$. Then:
\begin{align*} \pdx{i} \pdx{n} (Eu) (x_0) &= \pdx{i} \pdx{n}^+ (Eu) (x_0) \\
&= (\pdx{i} \pdx{n}^+ u) (x_0) \\
&= (u_{in})(x_0) 
\end{align*}
\begin{align*} (\pdx{n}^- \pdx{i} (E u)) (x_0) &= (\pdx{n}^- (\pdx{i} u + \pdx{n} u (\pdx{i} \frac{b_n}{a_{nn}}) x_n^2) \circ F) (x_0) \\
&= \pdx{n}^- ((\pdx{i} u) \circ F) (x_0) \notag \\ 
&= -(\pdx{n}^+ \pdx{i} u)(F'(x_0))=( \pdx{n}^+ \pdx{i} u) (x_0) \notag
\end{align*}
Hence $Eu$ is twice differentiable in $[-R,R]^{(n-1)} \times (-\delta,\infty)$.   
The second derivatives are also Hölder continuous in $x_n=0$ since they are Hölder continuous in $Q^+$ and $Q^-$.
Choose now a cutoff $\eta$ for $\R_0^+$ and $(-\delta,R]$ and define finally:
$$ E_1 u = \eta(x_n)  Eu $$
Then $E_1 : X_1 \longrightarrow \Ctwoalpha{Q}$, analogously to \ref{Exten1dim} one can show using \ref{ContLinOp}, that the operator is also continuous from \\ $(X_1,\norm{\cdot}{\Ctwoalpha{Q^{0,+}}})$ into $(\Ctwoalpha{Q},\norm{\cdot}{\Ctwoalpha{Q}})$
\end{proof}

\newcommand{\Ux}{U_{x_0}}
\newcommand{\Utilde}{\widetilde{U}_{x_0}}
\newcommand{\Utildei}[1]{\widetilde{U}_{x_{#1}}}
\newcommand{\Vtildei}[1]{\widetilde{V}_{y_{#1}}}
\newcommand{\Vx}{\widetilde{V}_{y_0}}
\newcommand{\Psix}{\Psi_{x_0}}
\newcommand{\V}{V}
\newcommand{\Vplus}{V^+}
\newcommand{\Vminus}{V^-}
\newcommand{\Q}{Q}
\newcommand{\Qplus}{\Q^+}
\newcommand{\Qminus}{\Q^-}

\newcommand{\Partuni}{\Phi}
\newcommand{\tildeU}{\tilde{U}}
\begin{theorem}[Extension on whole domain] { \ } \label{Exten} \\
Let $\Om \subset \R^n$ be a bounded domain with $C^{4,\alpha}$-smooth boundary. $L$ the elliptic differential operator $\eqref{defG}$ with coefficients $a_{ij}$,$b_{i} \in \Ctwoalpha{\Omclo}$.
The space $\{u \in \Ctwoalpha{\Omclo} \ | \ u=0,\ Lu=0 \ \text{on} \ \partial \Om \} $  can be continuously and linearly embedded into $\Ctwoalphac{\R^n}$ such that $\norm{u}{C^0_c(\R^n)} = \norm{u}{C^0(\Omclo)}$
\end{theorem}
\begin{proof}
The proof works as follows: For each point the boundary can be flatten out locally. For the flat boundary the diffeomorphism for transforming the differential operator (\ref{TrafoNoMix}) and the extension operator (\ref{ExtenStraight}) are applied to extend the function locally. The global construction follows using a partition of unity.
Now to the details:
Let $R>0$. By definition of $C^{4,\alpha}$-smoothness we find for each point $x_0 \in \partial \Om$ a neighborhood $U_{x_0}$ and a diffeomorphism $\Psix$ such that $\Psix(\Ux) = \BR$, $\Psix(\Ux \cap \partial \Om) \subset \BRzero$, $\Psix(\Ux \cap \Om) \subset \BRplus$, i.e. points inside $\Om$ have positive component $x_n$, points on the boundary zero component etc. 
Now fix some $x_0 \in \partial \Om $
Denote by $L_{1,x_0} = T_{\Psi_{x_0}}^{-1} \circ L  \circ T_{\Psi_{x_0}} $ the transformed differential operator acting on $\Ctwoalphaf{B^{0,+}_R(0)}$. This transformed differential operator has also $C^{2,\alpha}$ smooth first and second order coefficients by the remarks after \ref{TrafoDiffOp}.
By \ref{TrafoNoMix} there exists a $\Ctwoalphaf$-smooth diffeomorphism $\floone_{x_0}$ on $B_{R}(0)$ (after a possible substitution by some smaller $R'$). 
Then $L_2 = T^{-1}_{\floone_{x_0}} \circ L_{1,x_0} \circ T_{\floone_{x_0}}$ defines a differential operator having no mixed derivatives of second order on $\floone_{x_0}(\BRzero)$. 
Since $\V_{x_0}$ is open there exists a cuboid $\Q_{x_0} \subset \subset \V_{x_0}$.
Set $\tilde{U}_{x_0} = \Psi_{x_0}^{-1} (\floone_{x_0})^{-1} (Q_{x_0})$. The sets $(\tilde{U}_x)_{x \in \partial \Om}$ cover the boundary and by compactness of the boundary there exists a finite number $(\tilde{U}_i)_{1 \le i \le N}$ covering $\partial \Om$. Denote by $\Psi_i$,$\floone_i$,$\Q_i$, the corresponding diffeomorphism and cuboids. Let
\begin{enumerate}
\item $F_i = \floone_i \circ \Psi_i : \tildeU_i \rightarrow Q_i$
\item $T_i = T_{\Psi_i} \circ T_{\floone_i} : C^{s,\sigma}(\Q_i) \longrightarrow C^{s,\sigma}(\tilde{U}_i)$ for $s+\sigma \le 2+\alpha$, $s \in \N$, $0 < \sigma < 1$
\item $L^i_2 := T^{-1}_i \circ L \circ T_i \ : \Ctwoalpha{\Q^{0,+}_i} \longrightarrow \Czeroalpha{\Q^{0,+}_i} $
\end{enumerate}

Set $$X_i = \{ u \in \Ctwoalpha{Q^{0,+}_i} \ | \ u(z) = L^i_2 u(z) = 0 \ \text{for} \ z_n = 0, \  \} $$
Since $L^i_2$ and $X_i$ fulfil the assumptions of \ref{ExtenStraight} there exists a continuation operator $$ E_i : X_i \rightarrow \Ctwoalpha{\Q_i}$$ with $\norm{E_iu}{C^0{\Q_i}} = \norm{u}{C^0{\Qplus_i}}$. \newline
For $u \in \Ctwoalpha{\Omclo}$ with $Lu = u = 0$ on the boundary, we have $u_i:= T^{-1}_i u \in X_i$, indeed:
\begin{grothaus}
For $z \in Q^0$ we have $ F_i^{-1}(z) \in \partial \Om$ and thus $(T^{-1}_i u)(z) = u(F_i^{-1}(z)) = 0$ and $$L_2^i (u_i)(z) = (T_i^{-1} \circ L \circ T_i u_i)(z) = (L u)(F_i^{-1}(z)) = 0.$$  
And the smoothness is immediate by the smoothness of $u$ and the diffeomorphisms.
Thus the extension operator $E_i$ can be applied to $T^{-1}_i u$.
\end{grothaus}

Hence we can define a global extension operator in the following way:
Choose a $\tildeU_0 \subset \subset \Om$ with $\Om \subset \tildeU_0 \cup \bigcup \tildeU_i$ and the corresponding partition of unity $(\Partuni_i)^{N}_0$.  
$$ E : \{ u \in \Ctwoalpha{\Omclo} \ | \ u(x) = Lu (x) = 0 \ \text{for} \ x \in \partial \Om \} \longrightarrow \Ctwoalphac{\R^n}  $$
$$ u \longrightarrow  \Phi_0 u + \sumfromto{i}{N} \Phi_i T_i \circ E_i \circ T_i^{-1} u $$
where $\Phi_i$ is meant as the pointwise multiplication operator. \newline
Denote by $\tilde{E_i} = T_i\circ E_i \circ T_i^{-1}$ the local extension operator acting on $\tildeU_i$. Since for $x \in \Om \cap \tildeU_i$ it holds $\tilde{E_i} u (x) = u(x)$ and $\sumfromto{0}{N} \Phi_i =1$ in $\Om$ we have $Eu = u$ in $\Om$.
Since $$\sumfromto{0}{N} \Phi_i \le 1 \ \text{in} \ \Omclo^c$$ and $\Phi_i\ge 0$ and the local extension are contractive we also have: $$\norm{u}{\Czero{\R^n}} = \norm{u}{\Czero{\Om}}$$
Finally $\Phi_i \tilde{E_i} u \in \Ctwoalphac{\R^n}$ yields that $Eu \in \Ctwoalphac{\R^n}$. \newline
The continuity from $(X_1,\norm{\cdot}{\Ctwoalpha{\Omclo}})$ to $(\Ctwoalphac{\R^n},\norm{\cdot}{\Ctwoalphac{\R^n}})$ follows analougsly to \ref{ExtenStraight} using \ref{ContLinOp}.

\end{proof}

\begin{remark}

For $u \in \Ctwoalpha{\Om}$ the $\Ctwoalphaf$-norm of $Eu$ can be estimated by the $\Ctwoalphaf$-norm of $u$, of the function $F(a,b,s)$ (\ref{ReflFunc}), the diffeomorpism $\floone$ and $\Psi$ applied in \ref{Exten} and the partition of unity.

\end{remark}

\begin{scherz}
 \textit{Haben sie auch die Core-Property ??} (Grothaus)
\end{scherz}

\begin{theorem}[Core] { \ } \label{Core} \\
Let $\Om$ be a domain with $\Ctwoalphaf$-smooth boundary, $L$ the elliptic differential operator \ref{defG} with the additional property, that $a_{ij}$,$b_i \in \Ctwoalpha{\Omclo}$ (for $1\le i,j \le n$). 
Then the set 
$D(L) := \DL $ is a core for the generator of the operator semigroup $(T_t)_{t\ge 0}$, solving the Cauchy-Dirichlet Problem \eqref{CauchyProblem}, and can be embedded into $\Ctwoalphac{\R^n}$ such that $\| u \|_{C^0_c(\R^n)} = \| u \|_{C^0(\Omclo)} $
\end{theorem}

\begin{proof}
By \ref{EllipOpSemigroup} the closure of $(L,D(L))$ generates a strongly continuous semigroup $(T_t)$. By \ref{Exten} $D(L)$ can be embedded into $\Ctwoalphac{\R^n}$ in the required way.
\end{proof}

\section{Appendix}

For completeness we list here some theorems and lemmata which were used in this \whatthisis 

\subsection{Tools from elementary Analysis}

\begin{lemma}{Cutoff for A and B} { \ } \\
 Let $A$, $B$ two subsets of $\R^n$ with $A \subset \subset B$, i.e. $dist(A,\partial B)>0$.
 Then there exists a function $\eta$ called \textit{cutoff for A and B} with:
 $$\eta = 1 \ \text{on} \ A,\ \supp(\eta) \subset \subset B  ,\ \eta \in C^\infty(\R^n) $$
\end{lemma}
One can construct such an function easily by convolution of the indicator function of A with an approximate idenity with sufficiently small support, compare \cite{Alt},2.18, page 117.

\begin{definition}{Dirac sequence} { \ } \\
A sequence of functions $(\varphi_k)_{k \in \N}$ is called \textit{standard dirac sequence}, if:

\begin{enumerate}
 \item $\varphi_k \in C^\infty(\R^n)$
 \item $\supp (\varphi_k) \subset {B_{\frac{1}{k}}(0)}$
 \item $\varphi_k \ge 0$
 \item $\int_{\R^n} \varphi_k = 1$
\end{enumerate}

\end{definition}

\begin{lemma}{$C^\infty$-smoothness} { \ } \\
 Let $f \in L^1(\R^n)$, $\phi \in C_c^\infty(\R^n)$ then:
$$ \phi \ast f \in C^\infty(\R^n) $$
\end{lemma}

\begin{lemma}{$\Ckalpha$-norm estimate of the convolution} { \ } \label{ConvCkEstimate} \\
 Let $\Om$ be a bounded region, $B \subset \subset \Om$, $\varphi_k$ a standard dirac sequence with compact support, $u_k = \varphi_k \ast u$. \\
If $u \in \Ckalpha(\Omclo)$ then for all $k$ with dist$(B,\partial \Om) > \frac{1}{k}$.
$$ \norm{u_k}{\Ckalpha(B)} \le \norm{u}{\Ckalpha(\Omclo)}  $$
\end{lemma}

\begin{proof}
 First let $u \in C^{0,\alpha}(\Omclo)$. Set $C=\norm{u}{C^{0,\alpha}(\Omclo)}$. Then for $x \in B$, we have $x-z \in \Om$ for all $z \in B_{\frac{1}{k}}(0)$

 \begin{multline}
\frac{\mod{u_k(x) - u_k(y)}}{\mod{x-y}^\alpha} = \int_{B_{\frac{1}{k}}} \phi_k(z) \frac{\mod{u_k(x-z) - u_k(y-z)}}{\mod{x-y}^\alpha} dz \overset{*}{\le}  C \int_{B_{\frac{1}{k}}} \phi_k(z) dz = C
 \end{multline}

Where $*$ holds since $x-z$, $y-z \in \Om$ and thus 
$\mod{u_k(x-z) - u_k(y-z)} \le C \mod{x-y}^\alpha$

Now if $u \in C^k(\Omclo)$ a simple calculation involving an integration by parts yields:
$$D^s (\varphi_k \ast u) (x) = (\varphi_k \ast D^s u) (x) \ \text{for all} \ x \in B$$
for every multiindex with $\mod{s} \le k$. (See e.g. \cite{GilTru} page 143, Lemma 7.3).
Thus: $$\norm{D^s u}{C^0(B)} \le \norm{D^s u}{C^0(\Om)}$$
\end{proof}

\begin{lemma}{Extension of $\Ckalpha$-smooth boundary function} { \ } \label{TraceCk} \\
 Let $\Om$ be a region with a $\Ckalpha$-smooth boundary, let $g \in \Ckalpha(\partial \Om)$. Then there exists an $g_1 \in \Ckalpha(\Omclo)$ with $g_1 = g$ on $\partial \Om$.
\end{lemma}
See \cite{GilTru} page 131, Lemma 6.38
\begin{scherz}
 \textit{He doesn't like the word embedding} (Rosenberger)
\end{scherz}

\begin{lemma}{Smoothness of precompact sequences} { \ } \label{LimitPreCompact} \\
 Let $k \in \N$, $0 < \alpha \le 1$, $\Om$ a domain with $C^{k,\alpha}$-smooth boundary. If for some $l \in N$, $0 \le \beta \le 1$: $l+\beta < k +\alpha$, then the embedding:
$$ C^{k,\alpha}(\Omclo) \emb C^{l,\beta}(\Omclo) $$
is compact. 
Moreover a bounded sequence $u_n \in C^{k,\alpha}(\Omclo)$ has a subsequence converging in $C^{k}$-norm and the limit $u$ is again in $C^{k,\alpha}(\Omclo)$.
\end{lemma}
\begin{proof}
The proof of the compactness of the embedding can be found in \cite{GilTru}, page 130, Lemma 6.36.
Now the second part: Let $u_k \in \Czeroalpha{\Omclo}$ a sequence with $\norm{u_k}{\Czeroalpha{\Omclo}} \le C_\alpha$. By compactness there exists a $u \in C^0(\Omclo)$ and a subsequence with $u_{n_l} \overset{l \rightarrow \infty}{\longrightarrow} u$.
Now for $x,y \in \Omclo$ we have:
\begin{align*} \frac{\mod{u(x) - u(y)}}{\mod{x-y}^\alpha} &\le 2 \frac{\norm{u - u_{k_n}}{C(\Omclo)}}{\mod{x-y}^\alpha} + \frac{u_{k_n}(x) - u_{k_n}(y)}{\mod{x-y}^\alpha} g \\ 
&\le 2 \frac{\norm{u - u_{k_n}}{C(\Omclo)}}{\mod{x-y}^\alpha} + C_\alpha \notag
\intertext{By passing to the limit this gives:}
&\frac{\mod{u(x) - u(y)}}{\mod{x-y}^\alpha} \le C_\alpha
\end{align*}
Thus $u$ is also Hölder continous in $\Omclo$.
\end{proof}

\begin{theorem}[Stone-Weierstrass] { \ } \label{StoneWeier} \\
 Let $K$ be a compact set. Then $C^\infty(K)$ is dense in $C^0(K)$ w.r.t the sup-norm.
\end{theorem}

\begin{theorem}[Density of $C_c^\infty(\Omclo)$ functions in $C^0_0(\Omclo)$] { \ } \label{StoneWeierZero} \\
 Let $\Om \subset \R^n$ be a bounded domain. Then $C_c^\infty(\Omclo)$ is dense in $C_0^0(\Omclo)$ w.r.t the sup-norm
\end{theorem}
\begin{proof}
 Let $u \in C_0^0(\Omclo)$. Define $\Om_n = \{ x \in \Om \ | \ \text{dist}(x,\partial \Om) > \frac{1}{n} \}$. Thus $\Om = \bigcup \Om_n$ and \newline $\Om_n \subset \subset \Om_{n+1}$. 
 Let $\epsi > 0$. Since $u = 0$ on $\partial \Om$ there exists an $n$ such that: $ \mod{u(x)} \le \epsi$ for $x \notin \Om_n$.
$\clo{\Om_{n+1}}$ is compact, $u \lvert_{\clo{\Om_{n+1}}}$ is continous, by \ref{StoneWeier} there exists a $u_\epsi \in 
C^\infty(\clo{\Om_{n+1}})$ with $$\norm{u - u_\epsi}{C^0(\clo{\Om_{n+1}})} \le \epsi$$ 
In particular $u_\epsi(x) \le 2 \epsi$ for $x \in \partial \Om_n$.
Choose now a cutoff $\eta$ for $\Om_n$ and $\Om_{n+1}$. Define $\mtilde{u_\epsi} = \eta u_\epsi$. Then $\mtilde{u_\epsi} \in C_c^\infty(\Omclo)$ and $\mtilde{u_\epsi}(x) \le 2 \epsi$ for $x \notin \Om_n$.
For $x \in \Om_n$: $\mod{u(x) - \tilde{u_\epsi}(x)} = \mod{u(x) - u_\epsi(x)} = \epsi$.
\\ For $x \notin \Om_n$: $\mod{u(x) - \tilde{u_\epsi}(x)} \le \epsi + 2\epsi$.
Thus $\norm{u - \tilde{u_\epsi}}{C^0(\Omclo)} \le 3 \epsi$.
\end{proof}

\subsection{Functional analytic tools}
\begin{theorem}[Method of continuity] { \ } \label{TheoMethodCont} \\
 Let $X_1$, $X_2$ be two Banach spaces. $L_0$, $L_1$ linear continous operators from $X_1$ to $X_2$. 
 Set $L_t = (1-t) L_0 + t L_1$.
Assume there exists a constant $c > 0$ such that:
$$ \norm{L_t x}{X_2} \ge c \norm{x}{X_1} $$ for all $x \in X_1$, $t \in [0,1]$.
Then $L_0$ is surjective, iff $L_t$ is surjective for all $t \in [0,1]$.
\end{theorem}
The proof can be found in \cite{GilTru} p.70 Theorem 5.2 (Method of continuity).

\begin{theorem}[Continuity of a linear Operator] { \ } \label{ContLinOp} \\
 Let $(X,\norm{\cdot}{X})$, $(Y,\norm{\cdot}{Y})$ be two Banach spaces. Let $\tau_1$ be a coarser Hausdorff topology of $X$, $\tau_2$ a coarser Hausdorff topology of $Y$. (If $U \subset X$ is open w.r.t to $\tau_1$ then $U$ is also open w.r.t $\norm{\cdot}{X}$.)
Assume that $T: (X,\tau_1) \rightarrow (Y,\tau_2)$ is a continous linear operator. Then $T : (X,\norm{\cdot}{X}) \rightarrow (Y,\norm{\cdot}{Y})$ is also continous.
\end{theorem}
\begin{proof}
 We show, that the graph $\Gamma(T)$ is closed as a subspace of $(X,\norm{\cdot}{X}) \times (Y,\norm{\cdot}{Y})$. 
Let $(u_n)_n$ a sequence in $X$ with $(u_n) \overset{n \rightarrow \infty}{\longrightarrow} u \ \text{w.r.t} \ \norm{\cdot}{X}$ and $(L u_n) \overset{n \rightarrow \infty}{\longrightarrow} v \ \text{w.r.t} \ \norm{\cdot}{Y}$.
 Since $\tau_1$ and $\tau_2$ are coarser we have also:
\begin{eqnarray}
 (u_n) \overset{n \rightarrow \infty}{\longrightarrow} u \ \text{w.r.t} \ \tau_1 \nonumber \\
(L u_n) \overset{n \rightarrow \infty}{\longrightarrow} v \ \text{w.r.t} \ \tau_2 \label{convV}
\end{eqnarray}
Since $L$ is continous from $(X,\tau_1)$ to $(Y,\tau_2)$ we also have:
\begin{eqnarray}
 (L u_n) \overset{n \rightarrow \infty}{\longrightarrow} Lu \ \text{w.r.t} \ \tau_2 \label{convLu}
\end{eqnarray}

Since $\tau_2$ is hausdorff \eqref{convV} and \eqref{convLu} imply: $Lu = v$. Thus $\Gamma(L)$ is closed and therefore by the Closed-Graph theorem $L$ is continous from $(X,\norm{\cdot}{X})$ to $(Y,\norm{\cdot}{Y})$.

\end{proof}

\subsection{Operator semigroup theory}

Throughout this subsection $X$ is a Banach space.
\begin{definition}{Strongly continous (contraction) semigroup } { \ } \label{OpSemigroup} \\
 A family of operators $T_t$, $t \in (0,\infty)$ fulfilling:
  \begin{enumerate}
\item $T_t u \overset{t \rightarrow 0}{\longrightarrow} u \ \text{for all } u \in X$
   \item $T_{t+s} = T_t T_s$
  \end{enumerate}
 is called a \textit{strongly continous operator semigroup} (s.c.s). \\
 If $\norm{T_t}{op} \le 1$ then the family is called a \textit{strongly continous contraction operator semigroup} (s.c.c.s)
\end{definition}

\begin{remark}
One can show, that every strongly continous operator semigroup is exponentially bounded, i.e there exists an $\omega < \infty$ with:
$\norm{T_t}{op} \le \exp(\omega t)$
\end{remark}

\begin{definition}{Dissipative} { \ } \label{DefDissipative} \\
 A linear operator $(L,D(L))$ on $X$ is called dissipative, if for every $x \in D(L)$ there exists a $x' \in \dualitysetx$ with: $\dualitymap{x'}{Lx} \le 0$. 
\end{definition}
\begin{remark}
 Note that by Hahn-Banach there exists always an element $x' \in X'$ with $\dualitymap{x'}{x}   =\norm{x'}{X'}^2 = \norm{x}{X}^2$, thus the definition is senseful.
\end{remark}

\begin{lemma} 
 A linear operator $(L,D(L))$ is dissipative, iff for every $\lambda > 0$ and all $x \in D(L)$:
$$ \norm{(\lambda - L)x}{} \ge \lambda \norm{x}{} $$
\end{lemma}

\begin{theorem}[Hille-Yosida] { \ } \label{TheoHilleYosida} \\
 A linear operator $(L,D(L))$ is the generator of a strongly continous contraction semigroup, iff it fulfills the following properties:
  \begin{enumerate}
 \item $D(L)$ is dense in $X$.
 \item $(0,\infty) \subset \rho(L)$
 \item $\norm{R(\lambda,L)}{} \le \frac{1}{\lambda}$ for all $\lambda > 0$
\end{enumerate}

\end{theorem}

\begin{theorem}[Lumer-Phillips] { \ } \label{TheoLumerPhillips} \\
Let $(L,D(L))$ be a densely defined, closed and dissipative operator. If there exists a $\lambda_0 > 0$ with $rg(\lambda_0 -L) = X$ then $rg(\lambda -L)=X$ for all $\lambda>0$ and $(L,D(L))$ generates a strongly continous contraction semigroup.
\end{theorem}

\begin{theorem}[Closure of an essential m-dissipative operator] { \ } \label{TheoClosureOp}  \\
 Let $X$ be a Banach space, $(L,D(L))$ a densely defined dissipative operator. Then the following holds:
 \begin{enumerate}
  \item $(L,D(L))$ is closable
  \item $rg(\lambda - \clo{L}) = \clo{rg(\lambda - L)}$
  \item If there exists a $\lambda > 0$ such that: $\clo{rg(\lambda - L)} = X$ then the closure $(\clo{L},D(\clo{L}))$ generates a strongly continous contraction semigroup.
 \end{enumerate}
\end{theorem}

\begin{theorem}[Rescaling of a semigroup and generator] { \ } \label{TheoDissipOmega} \\
 Let $(L,D(L))$ be a densely defined operator. If there exists an $\omega > 0$ and an $\lambda > \omega$ such that the operator $\tilde{L} := L - \omega$ is dissipative and $\clo{rg(\lambda-L)} =  X $ then the closure of $(L,D(L))$ generates a strongly continous semigroup $(T_t)$ with $\| T_t \| \le \exp(\omega t)$.
\end{theorem}

\subsection{Tools from multidimensional Analysis}

\begin{lemma}[Differentiability of the Operator Inversion] { \ } \label{DiffInv} \\
 Let $A \in \R^{n \times n}$ with $Det(A) \ne 0$. Then there exists a neighborhood (w.r.t to the Operator topology) of $A$ such that the mapping $$Inv : V(A) \rightarrow GL(n)$$
$$ X \rightarrow X^{-1} $$ is holomorphic.
\end{lemma}
This can be seen using the Neumann-Series for $(A+X) = A(\Id + A^{-1}X)$ for $\norm{X}{op}$ sufficiently small.

\begin{theorem}[Local diffeomorphism theorem] { \ } \label{TheoDiff} \\
 Let $\Om \subset \R^n$ be an open set, $F : \Om \rightarrow \R^n$ a mapping, which is $\Ctwoalphaf$-smooth in $\Om$. Let $x_0 \in \Om$. If $Det (D F)(x_0) \ne 0$ then there exists a neighborhood $U$ of $x_0$, $V$ of $F(x_0)$ and a mapping $G : V \rightarrow U$ with $F \circ G = \Id_V$, $G \circ F = \Id_U$. Moreover the mapping is $\Ctwoalphaf$-smooth in $V$, i.e. $F$ is a $\Ctwoalphaf$-diffeomorphism.
\end{theorem}
\begin{proof}
By the well-known Inverse-Function theorem (see \cite{AmEsch}, Theorem 7.3, page 223) there exist a neighborhood $U$ of $(x_0)$, $V$ of $(F(x_0))$ and a $G \in C^1(V)$ such that $F : U \rightarrow V$ is bijective,  $ F \circ G = \Id_V$, $G \circ F = \Id_U$ and $(D F)(x) \in GL(n)$ for $x \in U$. 
Hence it is left to show that the first derivatives of $G$ are also differentiable and the second derivatives are Hölder continous.
We denote the variables in $V$ by $y$, and in $U$ by $x$.
We have $F \circ G (y) = y$ for $y \in V$. Thus $ \pdy{i} (F_j \circ G) = \delta_{ij} \ \text{for all } \ i \le n$ using the chain rule and product rule respectively we get for $x \in U$:
$$ \pdy{i} (F_j \circ G) (y)= \mysum{k}{n} \left(\pdx{k} F_j \right) \left( G(y) \right) (\pdy{i} G_k)(y) = \delta_{ij} \ \text{for all} \ j $$ \label{DerOneContDiff}
Thus: $ \pdy{i} G (y) = ((D F) (x))^{-1} e_i $. $F$ is twice differentiable and the Operator-Inversion is holomorphic, hence the first derivatives of $G$ are also differentiable.
\begin{multline} \pdy{i_1} \pdy{i_2} (F_j \circ G) (y) = \pdy{i_1} \mysum{k}{n} (\pdx{k} F_j )(G(y)) (\pdy{i_2} G_k)(y)  \\
= \mysum{k}{n} (\pdx{k} F_j)(G(y)) (\pdy{i_1} \pdy{i_2} G_k)(y) + \mysum{k}{n} \mysum{k_1}{n} (\pdx{k_1} \pdx{k} F_j) (G(y)) (\pdy{i_1} G_k)(y) (\pdy{i_2} G_{k_1})(y)  \\
= 0
\end{multline}
With $(D F)_{jk} = \pdx{k} F_j$
$$ (D F) (G(y)) (\pdy{i_1} \pdy{i_2} G)(y) = -\mysum{k}{n} \mysum{k_1}{n} (\pdx{k_1} \pdx{k} F)(G(y))( \pdy{i_1} G_k)(y) (\pdy{i_2} G_{k_1})(y) $$ 
thus:
$$ (\pdy{i_1} \pdy{i_2} G)(y) = ((D F) (G(y)))^{-1}( -\mysum{k}{n} \mysum{k_1}{n} (\pdx{k_1} \pdx{k} F)(G(y)) (\pdy{i_1} G_k)(y) (\pdy{i_2} G_{k_1})(y) ) $$ \label{DerTwo}

The terms $(\pdx{k_1} \pdx{k} F)(G(y)) (\pdy{i_1} G_k)(y) (\pdy{i_2} G_{k_1})(y)$ are Hölder continous since $F \in \Ctwoalphaf$ and $\pdy{i} G_k$ are continously differentiable. Thus the right hand side of \eqref{DerTwo} is Hölder continous.
\end{proof}

\begin{remark}
 One can easily see, that this procedure can be applied to higher derivatives for sufficiently smooth $F$. Thus if $F$ is $C^{k,\alpha}$ then $G$ is also $C^{k,\alpha}$-smooth.
\end{remark}

\begin{theorem}[Differentiability along the boundary and one sided limit] { \ } \label{DiffBoundary} \\
 Let $R>0$, $Q = (-R,R)^n \subset \R^n$ be a cuboid. Let $u \in C^{1}(Q^{0,+})$,
 denote by $u_i$ the continous extension of the partial derivative $\pdx{i} u$ for $i \le n$, and $u_{ij}$ the continous extension of the second partial derivative $i,j \le n$.
 then for $x_0 \in Q^0$ we have:
 \begin{enumerate} \label{enuOne}
  \item For $i \ne n$ $u$ is differentiable in the i-th coordinate and $(\pdx{i} u) (x_0) = u_i(x_0) \label{deri}$
  \item The one sided limit $\pdx{n}^+ u (x_0)$ exists and $\pdx{n}^+ u(x_0) = u_n(x_0) \label{dern}$
  \end{enumerate}
 If $u \in C^{2}(Q^{0,+})$, denote by $u_{ij}$ the continous extension of the second partial derivative $i,j \le n$, then for $i,j \ne n$:
 \begin{enumerate}
  \item The second partial derivative in directions $i$,$j$ exist and $(\pdx{i}\pdx{j} u) (x_0) = u_{ij}(x_0)$.
  \item $\pdx{n}^{+} u$ is differentiable in direction $i$ and $\pdx{i} \pdx{n}^+ u (x_0) = u_{in}(x_0)$
  \item $\pdx{n}^+ \pdx{i} u (x_0) $ exists and $\pdx{n}^+ \pdx{i} u (x_0) = u_{in}(x_0)$
 \end{enumerate}
 in particular we have $\pdx{i} \pdx{n}^+ u (x_0) = \pdx{n}^+ \pdx{i} u (x_0)$
\end{theorem}

\begin{proof}
Denote by $e_n$ the n-th unit vector, $u_n$ the continous extension of the interior derivative to the boundary. By virtue of the mean value theorem there exists $h' \in (0,h)$ such that:

$$ \frac{u(x_0 + he_n) - u(x_0)}{h} \overset{MVT}{=} \pdx{n} u(x_0 + h' e_n) = u_n (x_0 + h' e_n)  $$
For $h \rightarrow 0$ we have that $h' \rightarrow 0$ and thus since $u_n$ is continous in $Q^{0,+}$:
$$ \pdx{n}^+ u(x_0) = \underset{h \rightarrow 0}{\lim} u_n(x_0 + h' e_n) = u_n(x_0)$$
Now let $i \ne n$:
For $h$ small enough such that $x_0 + h e_i$ and $x_0 - h e_i$ are contained in $Q^0$ we have for every $\epsi>0$ 
\begin{multline*}
u(x_0 + he_i) -u(x_0) \\
=  u(x_0 +he_i) -u(x_0 +he_i + \epsi e_n) + u(x_0 +he_i + \epsi e_n) - u(x_0 + \epsi e_n) - u(x_0 + \epsi e_n) - u(x_0) 
\end{multline*}
By virtue of the mean value theorem there exists $\epsi_1,\epsi_2 \in (0,\epsi), h' \in (0,h)$:
\begin{multline*}
 = \pdx{n} u(x_0 + h e_i + \epsi_1 e_n)\epsi + \pdx{i} u(x_0 + h' e_i + \epsi e_n) h + \pdx{n} u(x_0 + \epsi_2 e_n) \epsi 
\end{multline*}
Thus:
\begin{multline*}
 \lvert \frac{u(x_0 + h e_i) - u(x_0)}{h} - u_i(x_0) \lvert \le \underset{\epsi \rightarrow 0}{\limsup} (\lvert  \frac{u_n(x_0 + h e_i + \epsi_1 e_n) + u_n(x_0 + \epsi_2 e_n)}{h}\epsi \lvert + \\
 \mod{ u_i(x_0 + h'e_i + \epsi e_n) - u_i(x_0)})    
\end{multline*}
Since $u_n$ is continous in $Q^{0,+}$, $u_n$ is bounded and therefore the first term of the right hand side converges to 0. Since $u_i$ is uniform continous we have:
$$
 \le \supr{h' \in (0,h)}{\mod{ u_i(x_0 + h'e_i) - u_i(x_0)} }
$$
which tends to zero for $h \rightarrow 0$.

Hence:
$$ \pdx{i} u(x_0) = \underset{\mod{h} \rightarrow 0}{\lim} \frac{u(x_0 + h e_i) - u(x_0)}{h} = u_i(x_0) $$

The statements for the second derivative follow by applying \ref{enuOne}.1 to $u_i$ and $u_n$ respectively, the third statement follows by applying \ref{enuOne}.2 to $u_i$. 
\end{proof}

\section{List of symbols}
\begin{enumerate}
 \item $\N,\R$ the set of natural numbers, the set of real numbers
 \item $C^{k,\alpha}(\Om)$ the space of k-times continously differentiable functions on a domain $\Om$
 \item $C^{k,\alpha}_c(\Om)$ the subspace of $C^{k,\alpha}(\Om)$ of functions with compact support in $\Om$ ($\Om$ a domain) 
\item $C^{k,\alpha}_0(\Om)$ the subspace of $C^{k,\alpha}(\Om)$ of functions, which are zero on the boundary of $\Om$ ($\Om$ a domain) 
 \item $C^{k,\alpha}(\clo{\Om})$ the space of all $C^{k,\alpha}(\Om)$-functions such that the derivatives admit a continous extension to the boundary and the second derivatives admit a hoelder continous extension. 
 \item $C_0(\R^n)$ the spaces of continous functions vanishing at infinity.
 \item $L^p(\R^n)$ the space of p-integrable measurable functions.
 \item $\dualitymap{\cdot}{\cdot}$ the dual pairing in Banach spaces
 \item $\dualityset$ the set of normalized tangent functionals
 \item $\eucp{}{}$ the euclidean scalar product on $\R^n$
 \item $\Hess(u)$ the Hessian matrix of a function u, $(H(u))_{i,j} = \partial_i \partial_j u$
 \item $ (D F) $ the Jacobi matrix of a differentiable map $F : \R^n \rightarrow \R^n$, $(DF)_{jk} = \partial_k F_j$
 \item $ GL(n) $ the group of all invertible $\R^n \times \R^n$-matrices
 \item $ e_i $ the i-th unit vector
 \item $ f \ast g$ the convolution of a $L^1(\R^n)$-function $f$ and a $L^p(\R^n)$-function $g$, i.e. $ (f \ast g)(x) = \int_{\R^n} f(x-y) g(y) dy$ 
 \item $A \subset \subset B$ for $B$ compact, $A$ is compact contained in $B$, i.e. $\clo{A} \subset B$
 \item $M^+ = M \cap (\R^{n-1} \times \R^+)$, for $M \subset \R^n$ 
 \item $M^- = M \cap (\R^{n-1} \times \R^-)$
 \item $M^0 = M \cap (\R^{n-1} \times {0})$
 \item $M^{0,+} = M \cap (\R^{n-1} \times \R^+_0)$
 \item $u^+$,$u^-$ the positive and negative part of a function respectively.
\end{enumerate}

\textbf{Acknowledgement:}

The author would like to thank:

\begin{itemize}
 \item Florian Conrad for uncountable helpful and interesting discussions and for the essential ideas for the extension operator
 \item Supervisor Prof. Dr. M. Grothaus for his kind support and patience
 \item All other members of the Functional Analysis group
 \item Prof. Dr. von Weizsäcker for interesting discussions
 \item Benedikt Heinrich for reviewing this thesis
 \item Paul Taylor for the package for drawing commutative diagramms
 \item Eduard Helly for the Theorem of Hahn-Banach
\end{itemize}

\label{InvFuncTheo}


%

\end{document}